\documentclass[12pt,a4paper]{article}
\scrollmode

\usepackage[utf8]{inputenc}
\usepackage[T1]{fontenc}

\usepackage[pdftex]{graphicx}

\usepackage{amsmath}
\usepackage{amssymb}
\usepackage{amsthm}
\usepackage{caption,subcaption}

\newcommand{\sE}{{\mathbb{E}}}
\newcommand{\sP}{{\mathbb{P}}}
\newcommand{\sN}{{\mathbb{N}}}

\newcommand{\sR}{{\mathbb{R}}}
\newcommand{\sY}{{\mathbb{Y}}}
\newcommand{\sT}{{\mathbb{T}}}

\newcommand{\GG}{{\mathcal{G}}}
\newcommand{\FF}{{\mathcal{F}}}
\newcommand{\DD}{{\mathcal{D}}}
\newcommand{\EE}{{\mathcal{E}}}
\newcommand{\HH}{{\mathcal{H}}}
\newcommand{\clock}{{\Xi}}
\newcommand{\ind}{{\mathbb{I}}}

\newcommand{\widebar}{\overline}

\newtheorem{teorema}{Theorem}[section]
\newtheorem{lema}[teorema]{Lemma}
\newtheorem{corolario}[teorema]{Corollary}
\newtheorem{observacao}[teorema]{Remark}
\newtheorem{proposicao}[teorema]{Proposition}
\newtheorem{definicao}{Definition}[section]

\newtheorem{suposicao}{Assumption}

\DeclareMathOperator{\sgn}{sgn}
\DeclareMathOperator{\real}{Re}

\author{Luiz Renato Fontes \footnote{IME-USP, Rua do Mat\~ao 1010, 05508-090
S\~ao Paulo SP,  Brazil, lrfontes@usp.br} \thanks{Partially
supported by CNPq grant 305760/2010-6, and FAPESP grant 2009/52379-8}
\and Gabriel R.~C.~Peixoto \footnote{IME-USP, Rua do Mat\~ao 1010, 05508-090
S\~ao Paulo SP,  Brazil, gabrielrcp@gmail.com} \thanks{Supported by
CNPq fellowship 140177/2012-4 and a CAPES institutional fellowship}}

\title{GREM-like K processes on trees\\ with infinite depth}

\begin{document}

\maketitle



\begin{abstract}
  
  We take up the issue of deriving the limit as $n\to\infty$ of the GREM-like K process on a tree with $n$ levels.
  Under specific
  conditions on the parameters of the process, implying the martingality of a modification of 
  the underlying clock process sequences, we obtain infinite level clock processes as nontrivial limits of the 
  finite level clocks, 
  and use them to construct a process on a suitable product space 
  which is then shown to be the 
  limit of the $n$ level K processes as $n\to\infty$. 
  Some properties of the limiting, infinite level K process are established, like an expression for the asymptotic
  empirical measure of cylinders, giving information on the prospective equilibrium measure of the infinite level dynamics.


\end{abstract}


\smallskip




\section{Introduction}
\label{chap:introduction}
The K process on a tree with finitely many levels/finite depth appeared in~\cite{kn:fgg} as the  
weak limit of trap models on a tree with finitely many levels as the {\em volume} of the tree 
diverges. With the appropriate choice of parameters, the trap model on a tree with finitely 
many levels  is a phenomenological model for a dynamics of the GREM (Generalized 
Random Energy Model~\cite{kn:d2}) at low temperature and under a regime of parameters 
and time scale such that each hierarchy\footnote{Let us recall that the GREM is a 
hierarchical mean-field spin glass, with an arbitrary fixed finite number of hierarchies.} 
is {\em close to equilibrium} under the dynamics.
It was introduced in~\cite{kn:sn}; see also~\cite{kn:b}. 
In this case it is called the GREM-like trap model, and the associated K process is called the 
GREM-like K process.
A Glauber dynamics for the REM (Random Energy Model~\cite{kn:d1}), namely the 
{\em random hopping dynamics}, was studied 
in~\cite{kn:abg}, and shown to exhibit {\em aging} on a long time scale far from 
equilibrium; \cite{kn:fl} shows that this dynamics, under a time rescaling where it is close
to equilibrium, converges to a (single level) K process~\cite{kn:fm}. Aging results for a 
random hopping dynamics for the {\em $p$-spin} model were derived in~\cite{kn:abc}; 
see also~\cite{kn:bg}. There are so far no published
study that we know of of Glauber dynamics for the GREM, but we expect that a properly
defined 
random hopping dynamics for the GREM with the proper parameters 
and at the right time scale also converges to the GREM-like K process~\cite{kn:fg}.

Let us briefly recall/describe the trap model on a tree with finitely many levels,
the associated K process, as well as the GREM-like versions. 
Consider a tree $\sT_n$ with $n$ levels/generations starting from a root $\emptyset$ at level 
$0$. Level 1 has volume $M_1$, and each vertex of level $i$ is connected to $M_{i+1}$ 
vertices of level $i+1$, $i=0,\ldots,n-1$. The trap model on a tree with $n$ levels is a 
Markov jump process on the leaves of $\sT_n$, which when at a leaf $x|_n$ waits an 
exponential time of mean $\gamma_n(x|_n)$ and jumps to another leaf $y|_n$, chosen as follows.
From $x|_n$, we go down the tree through the unique path connecting it to $\emptyset$, flipping 
coins found at each site along the way. The coin at site $x|_i$ on level $i$ has probability
$p_i(x|_i)$ of turning up heads, independently of all the other coins. We set $p_0(\emptyset)=0$,
and $p_i(x|_i)=(1+M_{i+1}\gamma_i(x|_i))^{-1}$, where $\gamma_i(x|_i)$,
$x|_i$ sites of $\sT_n$, $i=1,\ldots,n$, are positive parameters of the model. Let $x|_i$ be
the first site on the way from $x|_n$ to $\emptyset$ whose coin flip turns up tails. Then $y|_n$
is chosen uniformly among the leaves of $\sT_n$ which are descendants of $x|_i$. Provided the 
$\gamma$ parameters satisfy a summability condition (saying roughly that the sum over the 
leaves of $\sT_n$ of the products of the $\gamma$'s over the path from each leaf to $\emptyset$
converges as $M_i\to\infty$, $i=1,\ldots,n$), then (a suitable representation of) 
this process converges in distribution as $M_i\to\infty$, $i=1,\ldots,n$, to a process on 
$\widebar\sN_*^n$, 
where $\widebar\sN_*=\sN_*\cup\{\infty\}$, and $\sN_*=\{1,2,\ldots\}$ is the positive integers.
We call the limiting process the K process on an $n$ level tree with set of parameters
$\{\gamma_i(x|_i),\,x|_i\in\sN_*^i,\,i=1,\ldots,n\}$.

It is the goal of this paper to derive a nontrivial 
limit of the latter process 
as $n\to\infty$ in the case of the GREM-like K process.
This process is a K process on 
a finite level tree, characterized by the following choice of parameters.  Let $n$ be
the depth/number of levels of the tree.
For each $i=1,\ldots,n$ and each $x|_{i-1}\in\sN_*^{i-1}$ (with $\sN_*^{0}=\{\emptyset\}$),
let $\underline{\gamma_i}(x|_{i-1}):=\{\gamma_i(x|_i),\,x_i\in\sN_*\}$ be a Poisson point 
process on $(0,\infty)$ with intensity function $c_i t^{-1-\alpha_i}$ in decreasing order, 
independent of each other, where $x|_i=(x|_{i-1},x_i)$.
The constants $c_1,\ldots,c_n$ are positive, and for the moment arbitrary, and we
must have $0<\alpha_1<\ldots<\alpha_n<1$. As already briefly mentioned above, 
it arises as the scaling limit of the GREM-like
trap model of Sasaki and Nemoto~\cite{kn:sn}. This is a trap model on $\sT_n$ which 
under a suitable time rescaling, and provided the {\em volumes} $M_1,\ldots,M_n$ satisfy
a {\em fine tuning} condition among themselves, becomes a trap model on $\sT_n$ with
the following set of parameters. For each $i=1,\ldots,n$ and each $x|_{i-1}$ on the
$(i-1)$-th level of $\sT_n$, the $M_i$ descendants of $x|_{i}$ of $x|_{i-1}$ on level 
$i$ are \emph{i.i.d.} random variables in the basin of attraction of an $\alpha_i$-stable law,
suitably scaled (in the usual way, so that when ordered they converge in law to 
$\underline{\gamma_i}(x|_{i-1})$ as $M_i\to\infty$, $i=1,\ldots,n$). Detailed 
definitions and constructions of the K processes mentioned above are provided 
in the next section.

Let us now roughly explain the main steps and ideas of the derivation of our main result.
We will show that each coordinate of a suitable version of the $n$ level GREM-like K process
converges in probability. 
Suppose we are looking at its $k$-th coordinate. 
The appropriate $k$ level clock process is a key ingredient; let us call it $\theta_k^n$.
It is defined in terms of the composition of $n-k$ single level clock processes, and we want 
to take its limit as $n\to\infty$. We will do that by means of a martingale convergence theorem.
Since $(\theta_k^n)_n$ is {\em not} a martingale, we introduce a modification, namely  
$\tilde\theta_k^n$, which is. This modification is obtained by inserting {\em missing factors} 
in a certain way {\em at the end of} $\theta_k^n$. 
These missing factors, which depend on the random parameters only,
are themselves obtained via a martingale convergence theorem (this time the randomness comes 
from the parameters). In order that these martingale properties hold true and the limits 
are nontrivial, we need to make a specific choice of the 
constants $c_1,c_2,\ldots$ mentioned above, and require that $\alpha_n\to1$ as
$n\to\infty$ sufficiently fast. We then have limits for the modified clocks at all levels, and use them 
to define an infinite level process. After showing that the modifications introduced in the clocks
wash away in the limit, we are in position to argue directly that the infinite level process 
defined with the limiting modified
clocks is the limit of the original $n$ level K process as $n\to\infty$. 

We thus obtain an infinite level dynamics which is the limit as $n\to\infty$ of the $n$ level GREM-like 
K process
(under the appropriate assumptions on parameters). Combined with the convergence result of~\cite{kn:fgg},
by abstract nonsense, we have that this infinite level dynamics arises as the scaling limit of GREM-like
trap models as both volume and number of levels diverge (in a way which is however {\em not} specified 
by the abstract argument), provided of course that the right conditions are in place. Presumably 
this is also the case for suitable random hopping dynamics for the GREM under appropriate conditions.

At the closure of this introduction, we outline the organization of the remainder of this article.
In Section~\ref{sec:constr-main-result}, we define the GREM-like K process in detail and formulate
our main convergence result, namely Theorem~\ref{thm:conv-infinite-process}. 
Section~\ref{sec:preliminaries} is devoted to auxiliary results. In Section~\ref{sec:clocks} we
derive the limit of the clock processes, and in Section~\ref{sec:conv-k-proc} we prove 
Theorem~\ref{thm:conv-infinite-process}, obtaining along the way properties of the {\em infinity set}
of the limiting K process (that is, the set of times where any of its coordinates equals $\infty$). 
And in the final Section~\ref{sec:invar-meas-some},
we derive asymptotics for the empirical measure of cylinders of the limiting process, thus shedding 
light on the prospective equilibrium measure of the infinite level dynamics.



\section{Model and main result}
\label{sec:constr-main-result}

We start by defining the K process on a tree with finite depth via a
slight adaptation of the construction employed in \cite{kn:fgg}.  Many
elements of this construction will be used to define the infinite
depth version of this process.

The state space of the K process on a tree with depth $k$ is
$\widebar{\sN}_*^k$, where $\sN_* = \{1, 2, \ldots \}$ and
$\widebar{\sN}_* = \sN_* \cup \{ \infty \}$.  We denote elements of
this space by $x|_k = (x_1, \ldots, x_k)$. For notation brevity we
will often denote $x|_k = x_1x_2 \ldots x_k$ and also denote $x|_ky$
as the concatenation of $x|_k$ and $y$, that is $x|_k y = (x_1,
\ldots, x_k, y) \in \widebar{\sN}_*^{k+1}$.

It will be useful to visualize $\widebar{\sN}_*^k$ as the nodes at
depth $k$ of a tree, with $\emptyset$ as root and node $x|_k y$ as an
offspring of node $x|_k$, Figure \ref{fig:tree-structure} illustrates this
representation.

\begin{figure}[htb]
  \centering
  \includegraphics{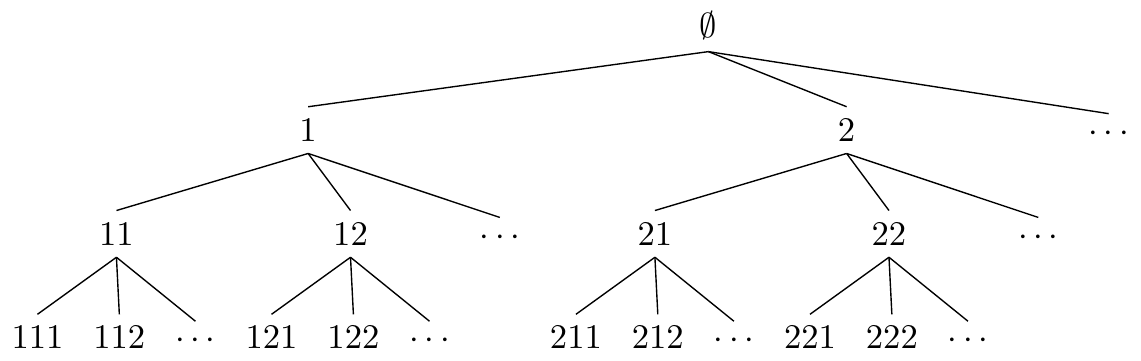}
  \caption{Tree representation of the state space}
  \label{fig:tree-structure}
\end{figure}

As parameters, take $0 < \alpha_1 < \alpha_2 < \ldots < 1$ real
numbers such that $\alpha_n \to 1$ as $n \to \infty$.

For each $k \in \sN_*$ and $x_{k-1} \in \sN_*^{k-1}$ take
$\gamma_k(x_{k-1} 1) > \gamma_k(x_{k-1} 2) > \ldots > 0$ the ordered marks of
a Poisson point process on $\sR^+$ with intensity measure $\mu_k$:
\begin{equation}
   \label{eq:constant-choice}
  \mu_k(d t) := \frac{c_k}{t^{1+\alpha_k}}, t > 0,
  \qquad
  c_k := \frac{\alpha_k}{\Gamma\left(1 - \frac{\alpha_k}{\alpha_{k+1}} \right)},
\end{equation}
where $\Gamma(\beta) = \int_0^\infty t^{\beta-1} e^{-t} d t$ is the
standard Gamma function. This choice for the constant $c_k$ was
deliberately made by us in order to obtain convergence of the
processes that we are about to construct, as pointed out at the Introduction.

The construction will be made on a probability space that admits all
these Poisson processes independently and independent from the
following random variables:
\begin{itemize}
\item $\{ T_i^{k, x}: i, k, x \in \sN_* \}$: a family of
  independent, identically distributed (\emph{i.i.d.}) random
  variables with exponential distribution of mean $1$;
\item $\{ N^{k,x}: k, x \in \sN_* \}$: a family
  of independent Poisson processes, each with  rate $1$. We will
  denote the marks of $N^{k,x}$ by $0 < \sigma^{k,x}_1 < \sigma^{k,
    x}_2 < \ldots$.
\end{itemize}

We will construct the K processes recursively.  Assume $X_0
\equiv \emptyset$ and define for $k = 1, 2, \ldots$ and $t \geq 0$:
\begin{subequations}
  \label{eq:finite-k-proc}
  \begin{equation}
    \label{eq:def-clock}
    \clock_k(t) := \sum_{x \in \sN_*} \sum_{i=1}^{N^{k,x}(t)}
    \gamma_k(X_{k-1}(\sigma_i^{k,x}) x)  T^{k, x}_i,
  \end{equation}
  \begin{equation}
    \label{eq:def-proc-k-finito}
    X_k(t) := \begin{cases}
      (X_{k-1}({\clock_k}^{-1}(t)), x),
      & \textrm{ if } t  \in \bigcup_{i=1}^{\infty} \left[
        \clock_k(\sigma^{k,x}_i -),
        \clock_k(\sigma^{k,x}_i)
      \right);\\
      (X_{k-1}({\clock_k}^{-1}(t)), \infty),
      & \textrm{ otherwise,}
    \end{cases}
  \end{equation}
\end{subequations}
where ${\clock_k}^{-1}(t) = \inf\{r \geq 0: \clock_k(r) > t\}$ is the
generalized inverse of $\clock_k$.

\begin{definicao}
  We call $X_k$ the {\em K process on a tree with depth $k$} or {\em with $k$ levels}, or still
  {\em $k$ level K process}, with parameter set 
$\underline{\gamma_k}:=\{\gamma_i(x|_i),\,x|_i\in\sN_*^i,\,i=1,\ldots,k\}$.
\end{definicao}

\begin{observacao}
  The processes $\clock_j$ are the so-called {\em clock processes}.  
\end{observacao}

\begin{observacao}
  Note that our choice for $\{\gamma_k(x|_k): k \in \sN_*, \, x|_k \in
  \sN_*^k \}$ satisfies almost surely:
  \begin{equation}
    \label{eq:sum-gamma}
    \sum_{x|_k \in \sN_*^k} \widebar{\gamma}_k(x|_k) < \infty,
    \qquad
    \widebar{\gamma}_k(x|_k) := \prod_{j=1}^k \gamma_j(x|_j).
  \end{equation}
  Any choice for $\{\gamma_k(x|_k)\}$, random or not, that satisfies
  this condition is suitable for the provided construction, as it
  guarantees that the clock processes are finite almost surely.
\end{observacao}

\begin{observacao}
  This version of the K process, with $\gamma_k(x|_k)$ given randomly as 
  specified above, was called GREM-like K process in
  \cite{kn:fgg}, where it is shown that, under suitable
  conditions, it is a scaling limit of the trap models introduced in
  \cite{kn:sn}.
  In order to prove this result, a slightly different
  construction is used, which is suited to the coupling argument there undertaken.
  We opted here for this slightly simpler
  construction because it will be more adequate to work with further
  on. Both constructions yield processes with the same law.
\end{observacao}

\begin{definicao}
  \label{def:theta}
  For $k \leq n$ define the {\em clock composition} as:
  \begin{equation}
    \label{eq:def-theta}
    \theta_k^n := \clock_n \circ \clock_{n-1} \circ \cdots \circ \clock_k.
  \end{equation}
  For simplicity, these stochastic processes will also be called
  clocks.
\end{definicao}

\begin{observacao}
  \label{obs:coordinate-k-process}
  Note that, for $j < k$, if $X_{k, j}$ is the j-th coordinate of
  $X_k$, the K process on a tree with depth $k$, then:
  \begin{equation}
    \label{eq:coordinate-k-process}
    X_{k, j}(t) = \begin{cases}
      x, & \text{ if } t \in \bigcup_{i=1}^\infty
      [\theta_j^k(\sigma_i^{j, x}-), \theta_j^k(\sigma_i^{j, x}));\\
      \infty, & \text{ otherwise.}
    \end{cases}
  \end{equation}
\end{observacao}

\begin{observacao}
  \label{obs:infinitos-niveis}
  We note that, since we did not define $\gamma_k(x|_k)$ for 
  $x|_k\in \widebar{\sN}_*^k \setminus \sN_*^k$, we had better
  show that $X_{k-1}(\sigma_i^{k, x}) \in \sN_*^{k-1}$ \emph{a.s.}.

  To do this, we first note that for $j < k$, $X_{k-1, j}(s) = \infty$ if
  and only if $s$ belongs to the image of $\theta_j^{k-1}$, and this set
  has null Lebesgue measure, since $\theta_j^{k-1}$ is a step function.

  Furthermore, we can infer from the construction that \emph{a.s.}, for
  any $j \leq k$, $X_{k, j}$ is constant and finite in any interval 
  $[\clock_k(\sigma_i^{k, x}-), \clock_k(\sigma_i^{k, x}))$.

  We also note that if $s$ is a discontinuity point of $\theta_k^n$, for
  $k \leq n$, then $s = \sigma_i^{k, x}$ for some $i, x \in\sN_*$. 
  This can be proven by noting that if $s$ is a discontinuity
  point of $\theta_k^n$ and $s \not\in \{ \sigma_i^{k, x}: i, x \in
  \sN_*\}$, then $\theta_k^{m-1}(s) = \sigma_j^{m, y}$ for some $k < m
  \leq n, y \in \sN_*$ and therefore $\sigma_j^{m, y}$ belongs to the
  image of $\theta_k^{m-1}$, event that has probability zero.
\end{observacao}

Remark \ref{obs:coordinate-k-process} suggests an approach to define
the K processes on a tree with infinite depth by taking the limit of
$\theta_k^n$ as $n \to \infty$. In Theorem \ref{thm:limit-clock} we
will define stochastic processes $\theta_k^\infty$, $k\geq1$, and in Theorem
\ref{thm:clock-conv} we will show that these stochastic processes are
in fact the limits of $\theta_k^n$, $k\geq1$, as $n \to \infty$. 
But the nontriviality of the limits requires a 
condition, namely 
\begin{equation}\label{cond2}
    \sum_{k=1}^\infty \frac{1-\alpha_{k+1}}{1-\alpha_k} < \infty.
\end{equation}
We will refer to this condition below as {\em the nontriviality condition}.
In order to
define the infinite level K process and state our main result, we will assume for
the remaining of this section that these processes $\theta_k^\infty$
are already constructed.

\begin{definicao}
  \label{def:k-proc-infinite}
  The {\em K process on a tree with infinite depth} or {\em with infinite levels},
  or still {\em infinite level K process}, with parameter set 
  $\underline{\gamma}:=\{\gamma_i(x|_i),\,x|_i\in\sN_*^i,\,i\geq1\}$,
  is a continuous time
  process $\sY = (Y_k)_{k \in \sN_*}$ taking values on
  $\widebar{\sN}_*^{\sN_*}$, where:
  \begin{equation}
    \label{eq:k-proc-infinite}
    Y_k(t) = \begin{cases}
      x, & \textrm{ if } t \in \bigcup_{i=1}^{\infty} \left[
        \theta_k^\infty(\sigma^{k, x}_i -),
        \theta_k^\infty(\sigma^{k, x}_i )
      \right);\\
      \infty, & \textrm{ otherwise.}
    \end{cases}
  \end{equation}
\end{definicao}

To be able to state a convergence theorem for the K processes in trees
with finite depth to the one on a tree with infinite depth, we need to
define an appropriate topology on the state space.  For a fixed
coordinate we will use a compactification of $\widebar{\sN}_*$. Namely
we equip $\widebar{\sN}_*$ with the metric $\rho_0$:
\begin{displaymath}
  \rho_0(x, y) := \left| \frac{1}{x} - \frac{1}{y}  \right|, \quad x, y
  \in \widebar{\sN}_*,
\end{displaymath}
under the convention that $\frac{1}{\infty} = 0$. In
$\widebar{\sN}_*^{\sN_*}$, we will adopt the metric $\rho$:
\begin{displaymath}
  \rho(x|_\infty, y|_\infty) := \sum_{k = 1}^{\infty} \frac{\rho_0(x_k, y_k)}{2^k}.
\end{displaymath}

Two points $x|_\infty$ and $y|_\infty$ are close in this metric if
they are close on a finite number or coordinates. What happens on
``large'' coordinates influences little.

We will extend the metric $\rho$ to $\bigcup_{k=1}^\infty
\widebar{\sN}_*^{k}$ by adding a new symbol $\zeta$ to the state space
and define $\rho_0(\zeta, x) := \ind\{ x = \zeta \}$. Then extend
$\rho$ by ``adding $\zeta$ at the end'' of $x|_k$. That is, for $j
\leq k \leq \infty$:
\begin{align*}
  \rho(x|_j, y|_k) &:=
 \sum_{i = 1}^{j} \frac{\rho_0(x_i, y_i)}{2^i} +
 \sum_{i = j+1}^{k} \frac{\rho_0(\zeta, y_i)}{2^i},\\
  \rho(y|_k, x|_j) &:= \rho(x|_j, y|_k). 
\end{align*}

\begin{observacao}
  $\rho$ is a complete metric over $\widebar{\sN}_*^{\sN_*} \cup
  \bigcup_{k=1}^\infty \widebar{\sN}_*^k$. It also generates a
  separable and compact topology.
\end{observacao}

\begin{teorema}
  \label{thm:conv-infinite-process}
  Under the nontriviality condition~(\ref{cond2}), $\sY$ is a càdlàg process under $\rho$
  and $X_k$ converges to $\sY$ as $k \to \infty$ in probability under
  the Skorohod topology using $\rho$.
\end{teorema}

\begin{observacao}
  As we shall see below, conditions~(\ref{eq:constant-choice}) 
  and~(\ref{cond2}) are crucial in our approach. Without them
  we cannot insure neither the existence nor the nontriviality of
  the limiting clocks $\theta_k^\infty$. See Remark~\ref{rmk:conds}
  below.
\end{observacao}

\begin{observacao}\label{gamma}
  We will denote by $\sP$ de underlying probability measure, with $\sE$ as expectation. We will use the notation $\sE_\gamma$  for the conditional
  expectation given $\underline{\gamma}$. 
\end{observacao}



\section{Preliminaries}
\label{sec:preliminaries}

In this section we will prove some auxiliary results on 
$\{\gamma_k(x|_k): x|_k \in \cup_{j=1}^\infty \sN_*^j \}$,
which may be thought of as a random environment for the process.

\begin{definicao}
  \label{def:cilindros}
  For a fixed $k, n \in \sN_*$, $k \leq n$, and $x|_k \in
  \widebar{\sN}_*^k$ let us denote the ``cylinders'' based on $x|_k$ as:
  \begin{align*}
    [x|_k]_n &:= \left\{
      y|_n \in \sN_*^n: y|_k = x|_k
    \right\},\\
    \widebar{[x|_k]}_n &:= \left\{
      y|_n \in \widebar{\sN}_*^n: y|_k = x|_k
    \right\}.
  \end{align*}
\end{definicao}

\begin{proposicao}
  \label{prop:ppp-invariance}
  Let $\{\gamma_i: i \in \sN \}$ be the marks of a Poisson point
  process on $\sR^+$ with intensity measure $\mu (d t) = c/t^{1+\alpha}$ for $t >
  0$ with some $\alpha \in (0, 1)$ and $c > 0$.  Let $\{ X_i : i \in
  \sN \}$ be \emph{i.i.d.} positive random variables, independent from
  the Poisson process, such that $\sE(X_1^\alpha) <
  \infty$. Then $\{ \gamma_i X_i : i \in \sN \}$ is also a Poisson
  Point process, with intensity measure $\sE(X_1^\alpha) \mu$.
\end{proposicao}

\begin{proof}
  Define $S := \{ (\gamma_i, X_i): i \in \sN\}$ and note that $S$ is
  the set of the marks of a Poisson point process on the first
  quadrant of $\sR^2$ with intensity measure $\pi = \mu \times \nu$,
  where $\nu$ is the probability measure of $X_1$.

  Note that $T(x, y) = xy$ is a continuous transformation without
  accumulation points outside of zero, so
  $\{ T(s): s \in S\} = \{\gamma_i X_i: i \in \sN\}$ is a Poisson
  point process (see eg.  \cite{kn:k}, Mapping Theorem, Section
  2.3). Its intensity measure can be computed as $\sE(X_1) \mu$.
\end{proof}

\begin{proposicao}
  \label{prop:stable-moments}
  Let $X$ be a positive random variable, with Laplace transform
  $\phi(\lambda) := \sE(e^{-\lambda X}) = e^{-c \lambda^\alpha}$, for some $c > 0$,
  $\alpha \in (0, 1)$. Then for $0 < \beta < \alpha$:
  \begin{displaymath}
    \sE(X^\beta) = c^{\beta/\alpha} 
    \frac{\Gamma\left(1 - \beta/\alpha \right)}{\Gamma(1-\beta)}.
  \end{displaymath}
\end{proposicao}

\begin{proof}
  Using Fubini's theorem and the fact that  $\phi^\prime(\lambda) = -
  \sE(X e^{-\lambda X})$ one can readily check that:
  \begin{displaymath}
    \sE(X^\beta) = - \frac{1}{\Gamma(1-\beta)} \int_0^\infty
    \phi^\prime(\lambda) \lambda^{-\beta} d \lambda
    = c^{\beta/\alpha} \frac{\Gamma\left(1 - \beta/\alpha
      \right)}{\Gamma(1-\beta)}.
    \qedhere
  \end{displaymath}
\end{proof}

\begin{proposicao}
  \label{prop:laplace-transform-gamma}
  For fixed $j, k \in \sN_*$, $j < k$, $x|_j \in \sN_*^j$ and $\lambda >
  0$:
  \begin{equation}
    \label{eq:laplace-transform-gamma}
    \sE \left[ \exp \left\{
        -\lambda \sum_{y|_k \in [x|_j]_k}
        \frac{\widebar{\gamma}_k(y|_k)}{\widebar{\gamma}_j(y|_j)}
    \right\} \right] =
  \exp\left\{ -
    \left[
      \frac{\Gamma(1-\alpha_k)}
      {\Gamma\left(1-\frac{\alpha_k}{\alpha_{k+1}}\right)}
    \right]^{\alpha_{j+1}/\alpha_k} \lambda^{\alpha_{j+1}}
  \right\}.
  \end{equation}

  Therefore $\sum_{y|_k \in [x|_j]_k}
  \frac{\widebar{\gamma}_k(y|_k)}{\widebar{\gamma}_j(y|_j)}$ has an
  $\alpha_{j+1}$-stable distribution and thus is finite \emph{a.s.}.
\end{proposicao}

\begin{proof}
  This computation will be done by induction on $j$. We omit the base,
  that is, when $j = k - 1$ because it can be done analogously as
  the induction step.

  Taking $0 \leq j < k - 1$, let us assume that
  \eqref{eq:laplace-transform-gamma} is true for $j + 1$, and show
  that it is also true for $j$.

  Using Campbell's Theorem (see eg. \cite{kn:k}, Section 3.2)
  and Propositions \ref{prop:ppp-invariance} and
  \ref{prop:stable-moments}, together with the induction hypothesis,
  we have that:
  \begin{align*}
    & \sE\left[ \exp \left\{
        -\lambda  \sum_{y|_k \in [x|_j]_k}
        \frac{\widebar{\gamma}_k(y|_k)}{\widebar{\gamma}_j(y|_j)}
      \right\} \right]\\ 
    = &
    \sE\left[ \exp \left\{
        -\lambda \sum_{x_{j+1}} \gamma_{j+1}(x|_{j+1})
        \sum_{y|_k \in [x|_{j+1}]_k}
        \frac{\widebar{\gamma}_k(y|_k)}{\widebar{\gamma}_j(y|_{j+1})}
      \right\} \right]\\
    = & \exp \left\{
      - \int_0^\infty  (1 - e^{-\lambda t})
      \sE \left[ \left(
          \sum_{y|_k \in [x|_{j+1}]_k}
          \frac{\widebar{\gamma}_k(y|_k)}{\widebar{\gamma}_j(y|_{j+1})}
        \right)^{\alpha_{j+1}} \right]
        \mu_{j+1} (d t)
    \right\}\\
    = & \exp\left\{ -
      \left[
        \frac{\Gamma(1-\alpha_k)}
        {\Gamma\left(1-\frac{\alpha_k}{\alpha_{k+1}}\right)}
      \right]^{\alpha_{j+1}/\alpha_k} \lambda^{\alpha_{j+1}}
    \right\}. \qedhere
  \end{align*}
\end{proof}

We are interested on the random variables treated on Proposition
\ref{prop:laplace-transform-gamma} and would like to be able to take
limits as $k \to \infty$. But we are unable to do so directly. Instead
we have the following proposition:

\begin{proposicao}
  \label{prop:def-W}
  For every $x|_k \in \sN_*^k$, the following limit exists almost
  surely:
  \begin{equation}
    \label{eq:def-W}
    W(x|_k) := \lim_{n \to \infty}
    \sum_{y|_n \in [x|_k]_n} \left(
      \frac{\widebar{\gamma}_n(y|_n)}{\widebar{\gamma}_k(y|_k)}
    \right)^{\alpha_{n+1}}.
  \end{equation}
  Moreover $W(x|_k)$ has an $\alpha_{k+1}$-stable distribution with
  Laplace transform:
  \begin{equation}
    \label{eq:laplace-w}
    \sE \left[ e^{-\lambda W(x|_k)} \right] = \exp \{-\lambda^{\alpha_{k+1}}\}.
  \end{equation}
\end{proposicao}

\begin{proof}
  Without loss of generality, let us assume $k = 0$ and define
  random variables $Z_n(\lambda)$ for $n > k$ and $\lambda > 0$ as:
  \begin{align*}
    Z_n(\lambda) := \exp \left\{ - \sum_{y|_n} \left(
      \lambda \widebar{\gamma}_n(y|_n)
    \right)^{\alpha_{n+1}} \right\}.
  \end{align*}

  Let us show that these variables are a martingale in relation to the
  filtration $(\DD_n)$, where $\DD_n$ is the $\sigma$-algebra
  generated by all Poisson point processes $\{\gamma_j(x|_j): j \leq
  n, \, x|_j \in \sN_*^j\}$. This is done again using Campbell's Theorem:
  \begin{align*}
    \sE\left( Z_{n+1}(\lambda) \middle| \FF_{n} \right) &=
    \sE \left[ \exp \left\{ - \sum_{y|_{n+1}}
        \left(\lambda \widebar{\gamma}_{n+1}(y|_{n+1})
        \right)^{\alpha_{n+2}} \right\} \middle\vert \DD_{n} \right]\\
    &=
    \prod_{y|_n}
    \sE \left[
      \exp \left\{
        - \left( \lambda \widebar{\gamma}_{n}(y|_{n})\right)^{\alpha_{n+2}}
        \sum_{y_{n+1}}
          \gamma_{n+1}(y|_{n+1})
          ^{\alpha_{n+2}}
      \right\}
      \middle| \DD_{n} \right]\\
    &=
    \prod_{y|_n}
    \exp \left\{
        - \int_0^\infty (1 - \exp\left(-
          \left(\lambda \widebar{\gamma}_{n}(y|_{n})\right)^{\alpha_{n+2}}
          t^{\alpha_{n+2}}\right))
        \mu_{n+1}(d t)
      \right\}\\
    &=
    \prod_{y|_n}
    \exp \left\{- \left( \lambda \widebar{\gamma}_{n}(y|_{n})\right)^{\alpha_{n+1}}
    \right\}\\
    &=
    \exp \left\{- \sum_{y|_n}\left(\lambda
        \widebar{\gamma}_{n}(y|_{n})\right)^{\alpha_{n+1}}
    \right\} = Z_n(\lambda)
  \end{align*}

  Since $(Z_n(1))$ is a positive martingale, using a martingale
  convergence theorem (see eg. \cite{kn:d}, Theorem 5.2.9) we conclude
  that $Z_n(1)$ converges \emph{a.s.}. Since $Z_n(1)$ converges, then its
  exponent must converge as well.

  The second claim can be obtained by using the previous result to
  explicitly compute:
  \begin{align}
    \label{eq:laplace-w-1}
    \sE \left[ \exp \left\{ - \lambda
        \sum_{y|_n} \widebar{\gamma}_n(y|_n)^{\alpha_{n+1}}
      \right\}
    \right]
    &=
    \sE \left[ Z_n(\lambda^{1/\alpha_{n+1}}) \right]\\
    \notag
    &=
    \sE \left[ Z_1(\lambda^{1/\alpha_{n+1}}) \right]\\
    \notag
    &= \exp\{ -\lambda^{\alpha_{1}/\alpha_{n+1}}\}
    \xrightarrow{n \to \infty} \exp\{-\lambda^{\alpha_{1}}\}.
    \qedhere
  \end{align}
\end{proof}

\begin{proposicao}
  \label{prop:composition-W}
  The family of random variables $\{ W(x|_k): x|_k \in
  \bigcup_{j=0}^\infty \sN_*^j \}$ satisfies a composition law. Namely
  for every $x|_k \in \sN_*^k$:
  \begin{equation}
    \label{eq:composition-W}
    W(x|_{k}) = \sum_{x_{k+1}} \gamma_{k+1}(x|_{k+1}) W(x|_{k+1})
    \textit{ a.s.}
  \end{equation}
\end{proposicao}

\begin{proof}
  Fix a realization such that \eqref{eq:def-W} is true for every
  $x|_k \in \cup_{j=1}^\infty \sN_*^j$ and fix an arbitrary $\epsilon > 0$, then:
  \begin{align}
    \notag
    W(x|_k) &= \lim_{n \to \infty} \sum_{y_n \in [x|_k]_n} \left(
      \frac{\widebar{\gamma}_n(y|_n)}{\widebar{\gamma}_k(x|_k)}
    \right)^{\alpha_{n+1}}\\
    \notag
    &= \lim_{n \to \infty}
    \sum_{x_{k+1}} (\gamma_{k+1}(x|_{k+1}))^{\alpha_{n+1}} \sum_{y_n \in [x|_{k+1}]_n} \left(
      \frac{\widebar{\gamma}_n(y|_n)}{\widebar{\gamma}_{k+1}(x|_{k+1})}
    \right)^{\alpha_{n+1}}\\
    \label{eq:composition-W-1}
    &= \lim_{n \to \infty}
    \sum_{\substack{x_{k+1}:\\ \gamma_{k+1}(x|_{k+1}) > \epsilon}}
    (\gamma_{k+1}(x|_{k+1}))^{\alpha_{n+1}} \sum_{y_n \in [x|_{k+1}]_n} \left(
      \frac{\widebar{\gamma}_n(y|_n)}{\widebar{\gamma}_{k+1}(x|_{k+1})}
    \right)^{\alpha_{n+1}}\\
    \label{eq:composition-W-2}
    &+ \lim_{n \to \infty}
    \sum_{\substack{x_{k+1}:\\ \gamma_{k+1}(x|_{k+1}) \leq \epsilon}}
    (\gamma_{k+1}(x|_{k+1}))^{\alpha_{n+1}} \sum_{y_n \in [x|_{k+1}]_n} \left(
      \frac{\widebar{\gamma}_n(y|_n)}{\widebar{\gamma}_{k+1}(x|_{k+1})}
    \right)^{\alpha_{n+1}}.
  \end{align}

  Let us treat terms these two terms separately and show that their
  respective limits exist in probability. For
  \eqref{eq:composition-W-1}, since the outermost sum is finite
  and $\alpha_n \to 1$ as $n \to \infty$, we have that:
  \begin{align*}
    \eqref{eq:composition-W-1} &=
    \sum_{\substack{x_{k+1}:\\ \gamma_{k+1}(x|_{k+1}) > \epsilon}}
    \gamma_{k+1}(x|_{k+1})
    \lim_{n \to \infty}
    \sum_{y_n \in [x|_{k+1}]_n} \left(
      \frac{\widebar{\gamma}_n(y|_n)}{\widebar{\gamma}_k(x|_{k+1})}
    \right)^{\alpha_{n+1}}\\
    &=
    \sum_{\substack{x_{k+1}:\\ \gamma_{k+1}(x|_{k+1}) > \epsilon}}
    \gamma_{k+1}(x|_{k+1}) W(x|_{k+1})
    \xrightarrow{\epsilon \to 0}
    \sum_{x_{k+1}}
    \gamma_{k+1}(x|_{k+1}) W(x|_{k+1}).
  \end{align*}

  Let $\DD'_k=\DD'_k(x|_k)$ be the $\sigma$-algebra generated by 
  $\{\gamma_{k+1}(x|_{k+1}) : x_{k+1} \in \sN_*\}$.  Using
  \eqref{eq:laplace-w-1}, we can compute the Laplace transform of
  \eqref{eq:composition-W-2} as:
  \begin{align*}
    \sE &\left[ \exp \left\{
        - \lambda
        \sum_{\substack{x_{k+1}:\\ \gamma_{k+1}(x|_{k+1}) \leq \epsilon}}
        (\gamma_{k+1}(x|_{k+1}))^{\alpha_{n+1}} \sum_{y_n \in [x|_{k+1}]_n} \left(
          \frac{\widebar{\gamma}_n(y|_n)}{\widebar{\gamma}_{k+1}(x|_{k+1})}
        \right)^{\alpha_{n+1}}
      \right\}
    \right]\\
    &=
    \sE \left[
      \prod_{\substack{x_{k+1}:\\ \gamma_{k+1}(x|_{k+1}) \leq \epsilon}}
      \sE \left[ \exp \left\{
          - \lambda
          (\gamma_{k+1}(x|_{k+1}))^{\alpha_{n+1}} \sum_{y_n \in [x|_{k+1}]_n} \left(
          \frac{\widebar{\gamma}_n(y|_n)}{\widebar{\gamma}_{k+1}(x|_{k+1})}
        \right)^{\alpha_{n+1}}
      \right\}
    \middle| \DD'_k \right] \right]\\
    &=
    \sE \left[
      \prod_{\substack{x_{k+1}:\\ \gamma_{k+1}(x|_{k+1}) \leq \epsilon}}
      \exp\left\{
          - \left( \lambda
          (\gamma_{k+1}(x|_{k+1}))^{\alpha_{n+1}} \right)^{\frac{\alpha_{k+2}}{\alpha_{n+1}}}
      \right\}
    \right]\\
    &\xrightarrow{n \to \infty}
    \sE \left[
      \exp\left\{
        -  \lambda^{\alpha_{k+2}}
        \sum_{\substack{x_{k+1}:\\ \gamma_{k+1}(x|_{k+1}) \leq \epsilon}}
        (\gamma_{k+1}(x|_{k+1}))^{\alpha_{k+2}}
      \right\}
    \right]\\
    &= \exp\left\{
      - \int_0^\epsilon \left(
        1 - e^{-{(\lambda x)}^{\alpha_{k+2}}}
      \right) \mu_{k+1}(d x)
    \right\} \xrightarrow{\epsilon \to 0} 1.
  \end{align*}

  We used Campbell's Theorem for the last passage. With this we can
  conclude that the quantity in \eqref{eq:composition-W-2} converges
  in probability to $0$ as $\epsilon \to 0$.

  Finally, note that we have shown that \eqref{eq:composition-W-1} $+$
  \eqref{eq:composition-W-2} converge in probability to both $W(x|_k)$
  and $\sum_{x_{k+1}} \gamma_{k+1}(x|_{k+1}) W(x|_{k+1})$ as $\epsilon
  \to 0$, which implies that these last quantities are almost surely
  equal.
\end{proof}



%
\section{Limiting clocks}
\label{sec:clocks}

Our objective in this section is to define the limiting clocks
$\theta_k^\infty$ mentioned in Section
\ref{sec:constr-main-result}. For this purpose we will introduce a
perturbation to the clocks using the variables $W(x|_k)$ introduced in
Proposition \ref{prop:def-W}. Theorem \ref{thm:clock-conv} of the next
section will prove that this perturbation does not affect the limit.

\begin{definicao}
  \label{def:new-clock}
  We define the adjusted clocks as:
  \begin{align}
    \label{eq:new-clock}
    \widetilde{\clock}_j(t) &:= \sum_{x \in \sN_*} \sum_{i=1}^{N^{j,x}(t)}
    W(X_{j-1}(\sigma_i^{j,x}) x)
    \gamma_j(X_{j-1}(\sigma_i^{j,x}) x)
    T^{j,x}_i\\
    \notag
    \widetilde{\theta}_k^n &:= \widetilde{\clock}_n \circ \clock_{n-1}
    \circ \cdots \circ \clock_k.
  \end{align}
  Let us also denote the time spent by the K process $X_k$ on a state
  $x|_k$ up to time $t$ as $L_k(x|_k, t)$, that is:
  \begin{align}
    \label{eq:tempo-estado}
    L_k(x|_k, t) = \int_0^t \ind\{X_k(s) = x|_k\} d s.
  \end{align}
\end{definicao}

For almost every fixed environment 
$\{\gamma_k(x|_k): k \in \sN_*, x|_k \in \sN_*^k \}$, we have that 
$\theta_1^n$ is a subordinator (we will prove this on Lemma \ref{lema:theta-subordinator}), 
but $\theta_k^n$, for $k > 1$, is not. This creates some complications. 
To circumvent most of these problems we will break $\theta_k^n$ into a sum of functions that
are subordinators.

\begin{definicao}
  \label{def:pedacos-clock}
  For fixed $n > k \geq 1$, let $\nu_{k+1}^n$ be the random measure on the
  Borel sets of $[0, \infty)$ such that $\nu_{k+1}^n([0, t]) =
  \theta_{k+1}^n(t)$.  For a fixed $x|_k \in \sN^k_*$, we define:
  \begin{displaymath}
    \theta_{x|_k}^n(t) :=  \nu_{k+1}^n \left( \left\{
        s \in [0, \infty): L(x|_k, s) \leq t, \, X_k(s) = x|_k
      \right\} \right).
  \end{displaymath}
  Figure \ref{fig:pedacos-clock} illustrates this definition. The
  intervals marked on the abscissa of Figure \ref{fig:pedacos-clock-1}
  are the ones where $X_{k} = x|_k$.
\end{definicao}

\begin{figure}[htb!]
  \centering
  \begin{subfigure}[b]{0.45\textwidth}
    \centering
    \includegraphics{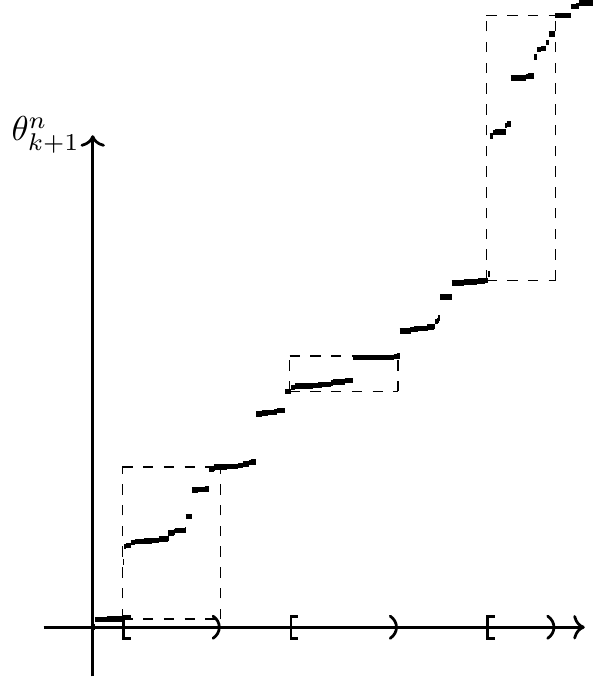}
    \caption{A realization of $\theta_{k+1}^n$}
   \label{fig:pedacos-clock-1}
  \end{subfigure}
  \begin{subfigure}[b]{0.45\textwidth}
    \centering
    \includegraphics{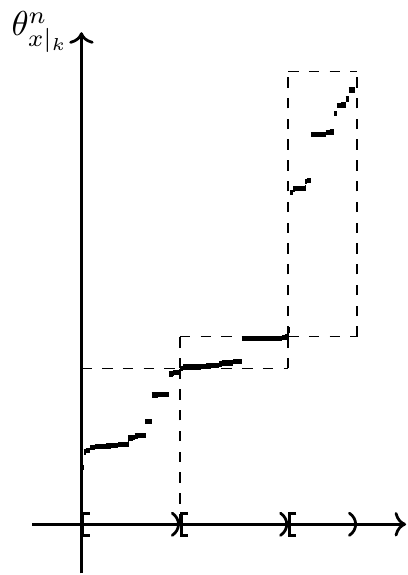}
    \caption{How to construct $\theta_{x|_k}^n$ from $\theta_{k+1}^n$}
   \label{fig:pedacos-clock-2}
  \end{subfigure}
  \caption{Construction of $\theta_{x|_{k}}^n$}
  \label{fig:pedacos-clock}
\end{figure}

\begin{observacao}
  \label{obs:pedacos-clock}
  Note that $\theta_{x|_k}^n$ has the same law as $\theta_1^{n-k}$,
  but with the index of the $\alpha$'s shifted, that is, it has the
  same law as a process $\widehat{\theta}_1^{n-k}$, constructed the
  same way as $\theta_1^{n-k}$, but with parameters
  $(\widehat{\alpha}_i)$, where $\widehat{\alpha}_i = \alpha_{i+k}$.

  These processes are also independent for a fixed $n$ and varying
  $x|_k$. And it is true that:
  \begin{equation}
    \theta_{k+1}^n (t) = \sum_{x|_k} \theta_{x|_k}^n (L_k(x|_k, t))
    \textit{ a.s.}
  \end{equation}
\end{observacao}

\begin{proposicao}
  \label{prop:laplace-clock}
  For $n > k \geq 0$, the Laplace transforms of $\theta_{k+1}^n(t)$ and
  $\widetilde{\theta}_{k+1}^n(t)$ satisfy
 \begin{align}
    \notag
    &\sE \left[ \exp\left\{
        - \lambda \theta_{k+1}^n(t)
      \right\} \right]\\ &=
    \sE \left[ \exp \left\{
        - \sum_{x|_k \in \sN^k_*} L_k(x|_k, t)
        \sum_{x_{k+1}} h_{x|_{k+1}} \left(
          \cdots
          \sum_{x_n} h_{x|_n} \left(\lambda \right)
          \cdots
        \right)
      \right\} \right],\\\notag&\\\notag
    &\sE \left[ \exp\left\{
        - \lambda \widetilde{\theta}_{k+1}^n(t)
      \right\} \right]\\\label{eq:laplace-clock} &=
    \sE \left[ \exp \left\{
        - \sum_{x|_k \in \sN^k_*} L_k(x|_k, t)
        \sum_{x_{k+1}} h_{x|_{k+1}} \left(
          \cdots
          \sum_{x_n} h_{x|_n} \left(\lambda W(x|_n)\right)
          \cdots
        \right)
      \right\} \right],
  \end{align}
   where
   $$h_{x|_k}(\lambda) := \frac{\lambda \gamma_k(x|_k)}{1 + \lambda \gamma_k(x|_k)},$$ 
  with the convention that $\sN^0_* = \{\emptyset\}$ and
  $L_0(\emptyset, t) = t$.
\end{proposicao}

\begin{proof}
  Let $\EE_n$ be the $\sigma$-algebra generated by all variables
  involved on the construction up to level $n$ and take $\{Z_{x|_n}:
  x|_n \in \sN_*^n \}$ an arbitrary family of positive independent
  random variables, this family will be assumed independent of
  $\EE_n$.

  Let us start by proving that:
  \begin{align}
    \label{eq:laplace-clock-1}
    \sE \left[ \exp\left\{
        - \sum_{x|_n} Z_{x|_n} L_n(x|_n, \theta_1^n(t))
      \right\} \right] &=
    \sE \left[ \exp \left\{
        -t \sum_{x_1} h_{x|_1} \left(
          \cdots
          \sum_{x_n} h_{x|_n} \left( Z_{x|_n} \right)
          \cdots
        \right)
      \right\} \right].
  \end{align}

  The case $k = 0$ will follow by taking $Z_{x|_n} = \lambda$ and
  $Z_{x|_n} = \lambda W(x|_n)$ for $\theta_1^n$ and $\widetilde{\theta}_1^n$
  respectively.

  Note that, because of the way that the K process was constructed in
  \eqref{eq:def-proc-k-finito}, and the fact that the increments of a Poisson
  Process are independent and stationary, then:
  \begin{align*}
    L_n(x|_n, \clock_n(t)) &= \sum_{i=1}^{N^{n, x_n}(t)}
    \ind\{ X_{n-1}(\sigma_i^{n, x_n}) = x|_{n-1} \}
    \gamma_n(x|_n) T_i^{n, x_n}\\
    & \substack{\DD \\ =}
     \sum_{i=1}^{N^{n, x_n}(L_{n-1}(x|_{n-1}, t))}
     \gamma_n(x|_n) T_i^{n, x_n}
  \end{align*}

  Letting $\FF_n := \sigma(\EE_{n-1}, Z_{x|_n}: x|_n \in \sN_*^n)$, we
  can compute:
  \begin{align*}
    \sE \left[ e^{- Z_{x|_n} L_n(x|_n, \theta_1^n(t))} \middle| \FF_n \right]
    &= 
    \sE \left[
      \exp\left\{
        -Z_{x|_n} \sum_{i=1}^{N^{n, x_n}(L_{n-1}(x|_{n-1}, \theta_1^{n-1}(t)))}
        \gamma_n(x|_n) T_i^{n, x_n}
      \right\}
      \middle| \FF_{n}
    \right]\\
    &=
    \exp\left\{
      - L_{n-1}(x|_{n-1}, \theta_1^{n-1}(t))
      h_{x|_n}(Z_{x|_n})
    \right\}
  \end{align*}

  Now let us return to \eqref{eq:laplace-clock-1}. We will prove that
  equality by induction on $n$. The base, $n = 1$, is obtained by
  simple inspection. Assuming the result to be true for $n - 1$, we
  can write:
  \begin{align*}
    &\sE \left[ \exp\left\{
        - \sum_{x|_n} Z_{x|_n} L_n(x|_n, \theta_1^n(t))
      \right\} \right]
    =
    \sE \left[  \sE \left[ \exp\left\{
          - \sum_{x|_n} Z_{x|_n} L_n(x|_n, \theta_1^n(t))
        \right\} \middle| \FF_n \right] \right]\\
    &=
    \sE \left[
      \prod_{x|_n}
      \sE \left[ \exp\left\{
          - Z_{x|_n} L_n(x|_n, \theta_1^n(t))
        \right\} \middle| \FF_n \right]
    \right]\\
    &= \sE \left[
      \exp\left\{
        - \sum_{x|_{n-1}} L_n(x|_{n-1}, \theta_1^{n-1}(t))
        \sum_{x_n} h_{x|_n}(Z_{x|_n})
      \right\}
    \right].
  \end{align*}

  Taking $Z_{x|_{n-1}} := \sum_{x_n} h_{x|_n} (Z_{x|_n})$, we can apply
  the induction hypothesis, obtaining \eqref{eq:laplace-clock-1}, from
  which we conclude the case $k = 0$.

  For the general case, using Remark \ref{obs:pedacos-clock}, we can write
  \begin{align*}
    \sE \left[
      e^{- \lambda \theta_{k+1}^n (t)}
    \right]
    &= \sE \left[
      \exp\left\{
        - \lambda \sum_{x|_k} \theta_{x|_k}^n (L_k(x|_k, t))
      \right\}
    \right]\\
    &= \sE \left[
      \prod_{x|_k}
      \sE \left[
        \exp\left\{
          - \lambda \theta_{x|_k}^n (L_k(x|_k, t))
        \right\}
        \middle|
        \EE_k
      \right]
    \right]\\
    &= \sE \left[
      \exp\left\{
        - \sum_{x|_k} L_k(x|_k, t) Z_{x|_k}^n
      \right\}
    \right].
  \end{align*}

  The corresponding result to $\widetilde{\theta}_{k+1}^n$ can be
  proven analogously.
\end{proof}

\begin{observacao}
  \label{obs:lebesgue-nula-integrado}
  For every $j \leq k$, note that the set $\{t \geq 0: \sP(X_{k, j}(t)
  = \infty) > 0 \}$ has null Lebesgue measure almost surely. This is a
  consequence of Fubini's Theorem and Remark
  \ref{obs:infinitos-niveis}:
  \begin{displaymath}
    \int_0^\infty  \sP(X_{k, j}(t) = \infty) d t
    = \sE \left[
      \int_0^\infty
      \ind \left\{
        X_{k, j}(t) = \infty
      \right\} d t
    \right] = 0.
  \end{displaymath}
\end{observacao}

\begin{teorema}
  \label{thm:limit-clock}
  For every $k \in \sN_*$ there exists almost surely a
  right-continuous non decreasing process $\theta_k^\infty$ such that
  $\lim_{n \to \infty} \widetilde{\theta}_k^n(t) = \theta_k^\infty
  (t)$ \emph{a.s.}, $t \in D_k$, where $D_k$ is a countable,
  deterministic and dense set of $[0, \infty)$. Furthermore, for
  almost every $\underline\gamma$, we have
  \begin{equation}
    \label{eq:expected-limit-clock}
    \sE_\gamma\left(\theta_{x|_k}^\infty(t)\right) \leq t W(x|_k).
  \end{equation}
\end{teorema}

\begin{proof}
  Let us fix $\underline\gamma$ and
  prove the result for almost all such choices.
  We also will denote by $\GG_k$ the $\sigma$-algebra generated by all
  Poisson processes and exponential random variables up to level $k$.

  Looking at definition \eqref{eq:new-clock}, we can split the ranges
  in which $\widetilde{\clock}_{k+1}$ uses the random environment from
  each of $x|_k$ of the previous levels, obtaining:
  \begin{align*}
    \sE_\gamma \left[ \widetilde{\clock}_{k+1}(t) \middle| \GG_k \right]
    &=
    \sum_{x|_k} L_k(x|_k, t) \sum_{x_{k+1}} W(x|_{k+1})
    \gamma_{k+1}(x|_{k+1})\\
    &=
    \sum_{x|_k} L_k(x|_k, t) W(x|_{k}).
  \end{align*}

  We can rewrite the definition in \eqref{eq:new-clock} to obtain:
  \begin{align*}
    \widetilde{\clock}_n(t) &= \sum_{x_n \in \sN_*} \sum_{i=1}^{N^{j,x_n}(t)}
    W(X_{n-1}(\sigma_i^{n,x_n}) x_n)
    \gamma_n(X_{n-1}(\sigma_i^{n,x_n}) x_n) T_i^{n, x_n}\\
    &=  \sum_{x|_n \in \sN_*^n} \sum_{i=1}^{N^{j,x_n}(t)}
    W(x|_n) \gamma_n(x|_n) \ind\{ X_{n-1}(\sigma_i^{n, x_n}) =
    x|_{n-1} \} T_i^{n, x_n}\\
    &=  \sum_{x|_n \in \sN_*^n} W(x|_n) L_n(x|_n, \clock_n(t))
  \end{align*}

  Finally taking $n > k$:
  \begin{align*}
    \sE_\gamma \left[ \widetilde{\theta}_k^{n+1}(t) \middle| \GG_n \right]
    &=
    \sE_\gamma \left[ \widetilde{\clock}_{n+1} (\theta_k^n(t)) \middle| \GG_n
    \right]\\
    &= \sum_{x|_n} L_n(x|_n, \theta_k^n(t)) W(x|_{n})\\
    &= \widetilde{\theta}_k^n(t)
  \end{align*}

  We have thus shown that, under $\sP_\gamma$, $(\widetilde{\theta}_k^n(t))_{n > k}$ 
  is a martingale with respect to the filtration $(\GG_n)_n$. Since it is
  nonnegative, we can again use a martingale convergence
  theorem (see eg. \cite{kn:d}, Theorem 5.2.9) to conclude that
  $\lim_{n \to \infty} \widetilde\theta_k^n(t)$ exists and is finite
  \emph{a.s.}~for each fixed $t \geq 0$. It also (along with Remark
  \ref{obs:pedacos-clock}) gives us \eqref{eq:expected-limit-clock}.

  Now to define the sets $D_k$ on the statement of the result, for $k
  = 1$ take $D_k$ the rational numbers of $[0, \infty)$. For $k > 1$,
  take a countable dense subset of $\{t \geq 0: \sP(X_{k-1}(t) \in
  \sN_*^{k-1}) = 1 \}$. This set is dense in $[0, \infty)$, since it
  has total Lebesgue measure (Remark
  \ref{obs:lebesgue-nula-integrado}).

  Let us take these limits $\theta_k^\infty(t) := \lim_{n \to \infty}
  \widetilde{\theta}_k^n(t)$ for $t \in D_k$. Since each
  $\widetilde{\theta}_k^n$ is monotonic, then the limit will be
  monotonic as well.

  With this we can define $\theta_k^\infty(t) = \lim_{s \to t+}
  \theta_k^\infty(s)$ for any $t \not\in D_k$, this limit being taken
  over $s \in D_k$.

  To complete the proof we only need to show that $\theta_k^\infty$ is
  right continuous over $D_k$. For a $t \in D_k$, let
  $\theta_k^\infty(t+) = \lim_{s \to t+} \theta_k^\infty(s)$, this
  limit exists almost surely because of monotonicity.

  We can compute the Laplace transform of $\theta_1^\infty(t+) -
  \theta_1^\infty(t)$ using Proposition \ref{prop:laplace-clock} and
  the fact that $\theta_1^n$ has stationary increments:
  \begin{align*}
    \sE & \left[
      \exp\left\{
        -\lambda \left(
          \theta_1^\infty(t+) - \theta_1^\infty(t)
        \right)
      \right\}
    \right] \\
    &= \lim_{s \to 0+} \lim_{n \to \infty}
    \sE \left[ \exp \left\{
        -s \sum_{x_1} h_{x|_1} \left(
          \sum_{x_2} h_{x|_2} \left(
            \cdots
            \sum_{x_n} h_{x|_n} \left(\lambda W(x|_n)\right)
            \cdots
          \right)
        \right)
      \right\} \right]\\
    &= \lim_{s \to 0+}
    \sE \left[ \exp \left\{
        -s \phi(\lambda)
      \right\} \right] = 1.
  \end{align*}

  The exchanges between limit and expected values taken here can be
  justified either by the continuity theorem for Laplace transforms or
  the dominated convergence theorem. Moreover this random function
  $\phi(\lambda)$ is finite \emph{a.s.} because $\theta_1^n(t)$ is
  finite \emph{a.s.}.

  So we have proved that $\theta_1^\infty(t+) = \theta_1^\infty(t)$
  \emph{a.s.} for every rational $t$. This concludes the proof for the
  case $k = 1$.

  For the case $k > 1$, note that our choice for $D_k$ guarantees
  that, with probability one, $X_{k-1}(t) \in \sN_*^{k-1}$ for all $t
  \in D_k$.

  Therefore $t$ belongs to an interval $[a, b)$ such that $X_{k-1}(s)
  = x|_{k-1}$ for all $s \in [a, b)$. So $L_{k-1}(y|_{k-1}, s)$ is
  constant in this interval for every $y|_{k-1} \neq x|_{k-1}$. Using
  Remark \ref{obs:pedacos-clock} e the previous case, we can conclude:
  \begin{align*}
    \theta_k^\infty(t+) - \theta_k^\infty(t)
    &= \lim_{s \to 0+} \lim_{n \to \infty}
    \theta_k^n(t+s) - \theta_k^n(t)\\
    &= \lim_{s \to 0+} \lim_{n \to \infty} \sum_{y|_{k-1}} \left[
    \theta_{y|_{k-1}}^n(L_{k-1}(y|_{k-1}, t+s)) -
    \theta_{y|_{k-1}}^n(L_{k-1}(y|_{k-1}, t))
    \right]\\
    &= \lim_{s \to 0+} \lim_{n \to \infty}
    \theta_{x|_{k-1}}^n(L_{k-1}(x|_{k-1}, t+s)) -
    \theta_{x|_{k-1}}^n(L_{k-1}(x|_{k-1}, t))\\
    &= \lim_{s \to 0+}
    \theta_{x|_{k-1}}^\infty(L_{k-1}(x|_{k-1},t+s)) -
    \theta_{x|_{k-1}}^\infty(L_{k-1}(x|_{k-1}, t))
    = 0.
    \qedhere
  \end{align*}
\end{proof}

\begin{teorema}[Non Triviality]
  \label{thm:nao-trivial}
  Suppose that:
  \begin{equation}
    \label{eq:nao-trivial}
    \sum_{k=1}^{\infty} \frac{1-\alpha_{k+1}}{1-\alpha_k} < \infty.
  \end{equation}
  Then, for every $k \in \sN_*$, $\theta_k^\infty$ is \emph{a.s.} a
  strictly increasing function and $\lim_{t \to \infty}
  \theta_k^\infty(t) = \infty$.
\end{teorema}

\begin{teorema}[Triviality]
  \label{thm:trivial}
  Suppose that:
  \begin{equation}
    \label{eq:trivial}
    \sum_{k=1}^{\infty} (1-\alpha_k) < \infty, \quad
    \sum_{k=1}^{\infty} \frac{1-\alpha_{k+1}}{1-\alpha_k} = \infty,
  \end{equation}
  then $\theta_k^\infty(t) = 0$ \emph{a.s.} for every $t \geq 0$ and
  $k \in \sN_*$.
\end{teorema}

\begin{observacao}
  Note that the case $\sum_k (1-\alpha_k) = \infty$ is not covered by
  either Theorem \ref{thm:nao-trivial} or \ref{thm:trivial}.
  We believe that the limit clocks are also trivial in this case
  (based on a few instances where we can see computations through,
  and on the intuition developed on these analyses; the general case
  seems too hard to compensate the expected dead end result).

  We will refer to condition \eqref{eq:nao-trivial} as the
  non-triviality condition. Note that it implies that $\sum_i
  (1-\alpha_i) < \infty$.
\end{observacao}

Both statements will be proven by studying the behavior of the random
variables in the exponent of the right hand side of
\eqref{eq:laplace-clock}. To make notations more compact, let us
define, for $n \geq k$ and a fixed $\lambda \geq 0$:
\begin{equation}
  \label{eq:exp-laplace-clock}
  Z_{x|_k}^n := \begin{cases}
    \lambda W(x|_k), & \text{ if } k = n;\\
    \sum_{x_{k+1}} h_{x|_{k+1}} ( Z_{x|_{k+1}}^n ), & \text{ otherwise.}
  \end{cases}
\end{equation}

Let us prove an auxiliary result that will be useful in the proof of
both theorems:
\begin{lema}
  \label{lema:di-somavel}
  Suppose that $\sum_i{1-\alpha_i} < \infty$. Then $\sum_i
  \frac{1-\alpha_{i+1}}{1-\alpha_i} < \infty$ if and only if
  $\sum_i (1-d_i) < \infty$, where:
  \begin{displaymath}
    d_i := \frac{\alpha_i \Gamma(\alpha_i)
      \Gamma(1-\alpha_i)}{\Gamma(1-\alpha_i/\alpha_{i+1})}.
  \end{displaymath}
\end{lema}
\begin{proof}
  Using that $\Gamma(1+\alpha) = \alpha \Gamma(\alpha)$ and developing
  the definition of $d_i$, we can write:
  \begin{align}
    \notag
    \Gamma &\left(2-\frac{\alpha_i}{\alpha_{i+1}} \right)
    \frac{(1 - d_i)(1-\alpha_i)}{1-\alpha_{i+1}}\\
    \label{eq:di-somavel-1}
    &=
    \frac{1-\alpha_i}{1-\alpha_{i+1}} \left[
      \Gamma\left(2 - \frac{\alpha_i}{\alpha_{i+1}} \right) -
      \Gamma(2-\alpha_i)
    \right]\\
    \label{eq:di-somavel-2}
    &+
    \frac{\Gamma(2-\alpha_i)}{1-\alpha_{i+1}} \left[
      1-\alpha_i - \left(1 - \frac{\alpha_i}{\alpha_{i+1}}\right)
    \right]\\
    \label{eq:di-somavel-3}
    &+
    \frac{\Gamma(2-\alpha_i)}{1-\alpha_{i+1}}
    \left( 1 - \frac{\alpha_i}{\alpha_{i+1}} \right)
    \left[
      1 - \Gamma(1+\alpha_i)
    \right].
  \end{align}

  Note that the three terms are positive for large enough $i$, since
  $\Gamma$ is decreasing near $1$ and increasing near $2$ and $0 <
  \alpha_i < \alpha_i/\alpha_{i+1} < 1$.

  Since $\Gamma$ is differentiable in $[1, 2]$, we have that
  \eqref{eq:di-somavel-1} converges to zero as $i \to \infty$. A
  straightforward computation shows that \eqref{eq:di-somavel-2}
  converges to $1$ as $i \to \infty$.

  Let $a_i := \eqref{eq:di-somavel-1} + \eqref{eq:di-somavel-2}$ and $b_i
  := \eqref{eq:di-somavel-3}$, we can write:
  \begin{align*}
    b_i \frac{1-\alpha_{i+1}}{1-\alpha_i} =
    \frac{\alpha_{i+1} - \alpha_i}{\alpha_{i+1}}
    \frac{\Gamma(2) - \Gamma(1+\alpha_i)}{1-\alpha_i}.
  \end{align*}
  Therefore $b_i^\prime := b_i \frac{1-\alpha_{i+1}}{1-\alpha_i}$ is summable, since
  $\sum_i (1-\alpha_i) < \infty$ and
  $\frac{\Gamma(2)-\Gamma(1+\alpha_i)}{1-\alpha_i}$ converges to a
  constant as $i \to \infty$.

  Finally, we can write:
  \begin{align*}
     (1-d_i)\Gamma\left( 2 - \frac{\alpha_i}{\alpha_{i+1}} \right)
     &= a_i \frac{1-\alpha_{i+1}}{1-\alpha_i}
     + b_i^\prime.
  \end{align*}

  Since $a_i \to 1$ and $\Gamma(2-\alpha_i/\alpha_{i+1}) \to 1$ as $i
  \to \infty$ and $b_i^\prime$ is summable, we conclude that
  $\sum_i (1-d_i) < \infty$ if and only if $\sum_i
  \frac{1-\alpha_{i+1}}{1-\alpha_i} < \infty$.
\end{proof}

\begin{proof}[Proof of Theorem \ref{thm:trivial}]
  We will only prove that $\theta_1^\infty \equiv 0$. It is
  straightforward to extend this case to the general case.

  Because of Proposition \ref{prop:laplace-clock} and
  \eqref{eq:exp-laplace-clock}, it is enough to prove that
  $Z_{\emptyset}^n$ converges to $0$ in probability as $n$ goes to
  infinity. We will show the stronger result that
  $\sE(Z_{\emptyset}^n) \to 0$ as $n \to \infty$.

  Let us define:
  \begin{displaymath}
    a_k^n = \begin{cases}
      \sE(Z_{x|_k}^n) & \text{ if } k < n,\\
      \sE\left(\left(Z_{x|_k}^n\right)^{\alpha_k}\right) & \text{ if } k = n.
    \end{cases}
  \end{displaymath}

  Using the last claim in Proposition \ref{prop:def-W}, together with
  Proposition \ref{prop:stable-moments}, we can compute $a_n^n =
  \lambda^{\alpha_n} \Gamma(1-\alpha_{n}/\alpha_{n+1}) /
  \Gamma(1-\alpha_n)$. Using Proposition \ref{prop:ppp-invariance}, together
  with Campbell's Theorem and Jensen's Inequality, we can write that
  for $k < n$:
  \begin{align}
    \notag
    a_{k-1}^n &= \sE \left[
      \sum_{x_k} \frac{\gamma_k(x|_k) Z_{x|_k}^n}{1+\gamma_k(x|_k) Z_{x|_k}^n}
    \right]\\
    \notag
    &= \int_0^\infty \frac{x}{1+x}
    \frac{\sE\left[ (Z_{x|_k}^n)^{\alpha_k}\right] c_k}
    {x^{1+\alpha_k}} d x\\
    \notag
    &\leq c_k (a_k^n)^{\alpha_k}
    \int_0^\infty \frac{1}{x^{\alpha_k} (1+x)} d x \\
    \notag
    &= c_k (a_k^n)^{\alpha_k}
    \int_0^1 y^{1-\alpha_k -1} (1-y)^{\alpha_k-1} dy\\
    \notag
    &= c_k (a_k^n)^{\alpha_k} \Gamma(1-\alpha_k) \Gamma(\alpha_k)\\
    \label{eq:trivial-1}
    &= (a_k^n)^{\alpha_k} \frac{\alpha_k \Gamma(\alpha_k) \Gamma(1-\alpha_k)}{\Gamma\left(1 - \frac{\alpha_k}{\alpha_{k+1}} \right)}.
  \end{align}

  In the last line we substituted the value of $c_k$ in
  \eqref{eq:constant-choice}. We can compute $a_{n-1}^n$ in an
  analogous way, but using the actual value of
  $\sE[(Z_{x|_n}^n)^{\alpha_n}]$ instead of estimating it via Jensen's
  inequality, obtaining:
  \begin{align}
    \label{eq:trivial-1-2}
    a_{n-1}^n &= a_n^n \frac{\alpha_n \Gamma(\alpha_n)
      \Gamma(1-\alpha_n)}{\Gamma\left(1 -
        \frac{\alpha_n}{\alpha_{n+1}} \right)}
    = \lambda^{\alpha_n} \alpha_n \Gamma(\alpha_n)
  \end{align}

  Iterating on \eqref{eq:trivial-1} and using this equality, we obtain
  that:
  \begin{align}
    \notag
    a_0^n
    &\leq \prod_{i=1}^{n-1} \left[ \frac{\alpha_i
        \Gamma(\alpha_i)
        \Gamma(1-\alpha_i)}{\Gamma\left(1-\frac{\alpha_i}{\alpha_{i+1}}\right)}
    \right]^{\alpha_1 \ldots \alpha_{i-1}}
    (a_{n-1}^n)^{\alpha_1 \ldots \alpha_{n-1}}\\
    \label{eq:trivial-2}
    &= (\lambda^{\alpha_n} \alpha_n \Gamma(\alpha_n))^{\alpha_1 \ldots \alpha_{n-1}}
    \prod_{i=1}^{n-1} \left[
      \frac{\alpha_i \Gamma(\alpha_i)
        \Gamma(1-\alpha_i)}{\Gamma\left(1-\frac{\alpha_i}{\alpha_{i+1}}\right)}
    \right]^{\alpha_1 \ldots \alpha_{i-1}}
  \end{align}

  Note that $\lambda^{\prod_j \alpha_j} \leq \max\{\lambda, 1\}$ and
  that $\alpha_n \Gamma(\alpha_n) \to 1$ as $n \to \infty$. So if we
  show that the product on \eqref{eq:trivial-2} converges to zero as
  $n \to \infty$, it will follow that $a_0^n \xrightarrow{n\to\infty}
  0$. This motivates the definition:
  \begin{align}
    \label{eq:trivial-3}
    d_i &:= \frac{\alpha_i \Gamma(\alpha_i)
      \Gamma(1-\alpha_i)}{\Gamma\left(1-\frac{\alpha_i}{\alpha_{i+1}}\right)},
    &
    b_i &:= d_i^{\alpha_1 \ldots \alpha_{i-1}}.
  \end{align}

  Now we want to show that $\prod_{i=1}^\infty b_i = 0$. Note
  that $d_i, b_i \in (0, 1)$, since $\Gamma$ is a decreasing function on
  $(0, 1)$ and $\alpha \Gamma(\alpha) = \Gamma(\alpha+1) < 1$ for $\alpha \in (0, 1)$.

  By assumption, $\sum_i (1 - \alpha_i) < \infty$. This implies that
  $\prod_{i} \alpha_i > 0$. Then:
  \begin{align*}
    \prod_i b_i = 0
    &\Leftrightarrow
    \sum_i \log b_i = - \infty
    \Leftrightarrow
    \sum_i \alpha_1 \ldots \alpha_{i-1} \log d_i = - \infty\\
    &\Leftrightarrow
    \sum_i \log d_i = - \infty
    \Leftrightarrow
    \prod_i d_i = 0
    \Leftrightarrow
    \sum_i (1 - d_i) = + \infty.
  \end{align*}

  Finally Lemma \ref{lema:di-somavel} guarantees that $\sum_i (1-d_i) =
  \infty$ whenever $\sum_i \frac{1-\alpha_{i+1}}{1-\alpha_i} = \infty$.
\end{proof}

Before proving Theorem \ref{thm:nao-trivial}, let us state an
auxiliary result:
\begin{lema}
  \label{lema:trivial-inferior}
  Let $\{\gamma_i: i \in \sN_*\}$ be the marks of a Poisson Process
  with intensity measure $\mu(dx) = \frac{c}{x^{1+\alpha}} \ind{\{x >
    0\}}$, for some $c > 0$ and $\alpha \in (0, 1)$. Taking $X :=
  \sum_i \frac{\gamma_i}{1+\gamma_i}$ and fixing an $\beta \in (0,
  1)$, it is true that:
  \begin{displaymath}
    \sE(X^\beta) \geq
    \frac{c \Gamma(\alpha) \Gamma(1-\alpha)}{\left[
        1 + c \Gamma(\alpha)\Gamma(1-\alpha)
    \right]^{1-\beta}}.
  \end{displaymath}
\end{lema}

\begin{proof}
  Take $\phi(\theta) := \sE(e^{-\theta X})$ the Laplace transform of
  $X$. We compute this quantity using Campbell's Theorem:
  \begin{align*}
    \phi(\theta) &= \exp\left\{
      - \int_0^\infty (1 - e^{- \theta \frac{x}{x+1}})
      \frac{c}{x^{1+\alpha}} d x
    \right\}\\
    &= \exp\left\{
      - c \int_0^1 \frac{1 - e^{- \theta y}}{y^{1+\alpha}(1-y)^{1-\alpha}} d y
    \right\}\\
    &=: \exp\left\{
      -c \psi(\theta)
    \right\}.
  \end{align*}

  The first derivative of $\psi$ can be computed and estimated as:
  \begin{align*}
    \psi^\prime (\theta)
    &= \lim_{h \to 0} \frac{\psi(\theta+h) - \psi(\theta)}{h}\\
    &= \lim_{h \to 0} \int_0^1 \frac{e^{-\theta
        y}}{y^{1+\alpha}(1-y)^{1-\alpha}}
    \frac{1-e^{-h y}}{h} d y\\
    &= \int_0^1 \frac{e^{-\theta y}}{y^{\alpha}(1-y)^{1-\alpha}} d y\\
    &\geq \int_0^1 \frac{e^{-\theta}}{y^{\alpha}(1-y)^{1-\alpha}} d
    y\\
    &= e^{-\theta}\Gamma(\alpha)\Gamma(1-\alpha),
  \end{align*}
  justifying the equality between the second and third line by the
  Dominated Convergence Theorem.

  Since $1-e^{-x} \leq x$, we can write:
  \begin{align*}
    \psi(\theta) &= \int_0^1 \frac{1 - e^{- \theta
        y}}{y^{1+\alpha}(1-y)^{1-\alpha}} d y\\
    &\leq \theta \int_0^1 \frac{1}{y^{\alpha}(1-y)^{1-\alpha}} d y\\
    &= \theta \Gamma(\alpha) \Gamma(1-\alpha)\\
    \phi(\theta) &= e^{-c \psi(\theta)}\\
    &\geq e^{- c \theta \Gamma(\alpha) \Gamma(1-\alpha)}.
  \end{align*}

  Finally we can conclude:
  \begin{align*}
    \sE(X^{\beta}) &= \sE(X^{1-(1-\beta)}) = -
    \frac{1}{\Gamma(1-\beta)} \int_0^\infty \theta^{(1-\beta)-1}
    \phi^\prime (\theta) d \theta\\
    &= \frac{1}{\Gamma(1-\beta)} \int_0^\infty \theta^{-\beta}
    c \phi(\theta) \psi^\prime (\theta) d \theta\\
    &\geq
    \frac{c}{\Gamma(1-\beta)} \int_0^\infty \theta^{-\beta}
    e^{- c \theta \Gamma(\alpha) \Gamma(1-\alpha)}
    e^{-\theta}\Gamma(\alpha)\Gamma(1-\alpha) d \theta\\
    &=
    \frac{c\Gamma(\alpha)\Gamma(1-\alpha)}
    {[1+c\Gamma(\alpha)\Gamma(1-\alpha)]^{1-\beta}}.
    \qedhere
  \end{align*}
\end{proof}

\begin{proof}[Proof of Theorem \ref{thm:nao-trivial}]
  Again we will show the claim only to $\theta_1^\infty$. We can
  extend the proof to the general case in a straightforward manner by
  using the fact that $\widetilde{\theta}_1^n =
  \widetilde{\theta}_{k+1}^n \circ \theta_1^{k}$.

  Let us define $a_k^n := \sE \left[ ( Z_{x|_k}^n)^{\alpha_k}
  \right]$. We will show that $\liminf_{n \to \infty} a_1^n >
  0$. Using Propositions \ref{prop:laplace-clock} and
  \ref{prop:ppp-invariance} we can write:
  \begin{align*}
    \sE \left[
      e^{-\lambda \widetilde{\theta}_1^n(t)}
    \right]
    = \sE \left[ \exp\left\{
      -t \sum_{x1} \gamma_1(x|_1)  Z_{x|_1}^n
    \right\} \right]
    = \sE \left[ \exp\left\{
      -t \sum_{x1} \gamma_1(x|_1)  (a_1^n)^{1/\alpha_1}
    \right\} \right].
  \end{align*}

  If we show that $\liminf_{n \to \infty} a_1^n > 0$, it will follow
  from the expression above that $\lim_{t \to \infty}
  \theta_1^\infty(t) = \infty$ in probability. Since $\theta_1^\infty$
  is \emph{a.s.} non decreasing, it follows that $\lim_{t \to \infty}
  \theta_1^\infty(t) = \infty$ \emph{a.s.}.

  To show that $\liminf_{n \to \infty} a_1^n > 0$, let us start by
  applying Lemma \ref{lema:trivial-inferior} and Proposition
  \ref{prop:ppp-invariance} to write that for $k \leq n$:
  \begin{align}
    \notag
    a_{k-1}^n 
    &\geq \frac{c_k a_k^n \Gamma(\alpha_k) \Gamma(1-\alpha_k)}{
      \left[
        1 + c_k a_k^n \Gamma(\alpha_k) \Gamma(1-\alpha_k)
      \right]^{1-\alpha_{k-1}}
    }\\
    \label{eq:nao-trivial-1}
    &= \frac{a_k^n
      \frac{\alpha_k \Gamma(\alpha_k)
        \Gamma(1-\alpha_k)}{\Gamma(1-\alpha_k/\alpha_{k+1})}
    }{
      \left[
        1 + a_k^n
      \frac{\alpha_k \Gamma(\alpha_k)
        \Gamma(1-\alpha_k)}{\Gamma(1-\alpha_k/\alpha_{k+1})}
      \right]^{1-\alpha_{k-1}}
    }
  \end{align}

  Let us first work with the denominator. For this look at
  \eqref{eq:trivial-2}. Although the definition of $a_k^n$ is
  slightly different on that proof, we can still use Jensen's
  inequality to obtain:
  \begin{align*}
    a_k^n &:= \sE\left[ (Z_{x|_k}^n)^{\alpha_k} \right]
    \leq \left( \sE\left[ Z_{x|_k}^n \right] \right)^{\alpha_k}\\
    &\leq \left(
      (\lambda^{\alpha_n} \alpha_n \Gamma(\alpha_n))^{\alpha_{k+1} \ldots \alpha_{n-1}}
      \prod_{i=k+1}^{n-1} \left[
        \frac{\alpha_i \Gamma(\alpha_i)
          \Gamma(1-\alpha_i)}{\Gamma\left(1-\frac{\alpha_i}{\alpha_{i+1}}\right)}
      \right]^{\alpha_{k+1} \ldots \alpha_{i-1}}
    \right)^{\alpha_k}\\
    &\leq \lambda^{\alpha_k \ldots \alpha_n}
    \leq max\{\lambda, 1\}
  \end{align*}
 
  Letting $\delta := \max\{\lambda, 1\}$ and using this expression on
  the denominator of \eqref{eq:nao-trivial-1}, we obtain:
  \begin{align}
    \label{eq:nao-trivial-2}
    a_{k-1}^n
    &\geq
    a_k^n \left( 
      1 + \delta
    \right)^{-(1-\alpha_{k-1})}
    \frac{\alpha_k \Gamma(\alpha_k)
      \Gamma(1-\alpha_k)}{\Gamma(1-\alpha_k/\alpha_{k+1})}.
  \end{align}

  Knowing that $0 < \alpha_i < \alpha_i/\alpha_{i+1} < 1$ and $\Gamma$
  is a decreasing function near zero, we can iterate this inequality
  to obtain:
  \begin{align}
    \notag
    a_k^n &\geq \lambda^{\alpha_n} \frac{\Gamma\left(1-\frac{\alpha_n}{\alpha_{n+1}}
      \right)}{\Gamma(1-\alpha_n)}
    (1+\delta)^{-\sum_{j=k}^{n-1} (1-\alpha_j)}
    \prod_{j=k+1}^n \frac{\alpha_j \Gamma(\alpha_j)
      \Gamma(1-\alpha_j)}{\Gamma\left(1-\frac{\alpha_j}{\alpha_{j+1}}
      \right)}\\
    \label{eq:nao-trivial-3}
    &\geq \lambda^{\alpha_n}
    (1+\delta)^{-\sum_{j=k}^{n-1} (1-\alpha_j)}
    \prod_{j=k+1}^n \frac{\alpha_j \Gamma(\alpha_j)
      \Gamma(1-\alpha_j)}{\Gamma\left(1-\frac{\alpha_j}{\alpha_{j+1}}
      \right)}.
  \end{align}

  Note that the terms in the product at the end of this last
  expression are exactly equal to $d_i$, defined in Lemma
  \ref{lema:di-somavel}. We have shown in that Lemma that, whenever
  $\sum_i \frac{1-\alpha_{i+1}}{1-\alpha_i} < \infty$, we have that $\sum_i
  (1-d_i) < \infty$, which implies that $\prod_i d_i > 0$.

  Therefore we conclude that $\liminf_{n \to \infty} a_k^n > 0$. To
  complete the proof of the theorem, we have to show that
  $\theta_1^\infty$ is \emph{a.s.} strictly increasing. We will do
  this by showing that $\lim_{\lambda \to \infty} \liminf_{n \to
    \infty} Z_\emptyset^n(\lambda) = \infty$ in probability.

  Since we are interested now only on large $\lambda$, we can assume
  $\lambda > 1$. Taking a constant $C > 0$ such that
  $\prod_{i=1}^\infty d_i > C$ and using the definition of $\delta$,
  we can rewrite \eqref{eq:nao-trivial-3} as:
  \begin{align*}
    a_k^n(\lambda)
    &\geq C \lambda^{\alpha_n}
    (1+\lambda)^{-\sum_{j=k}^{n-1} (1-\alpha_j)},\\
    \liminf_{n \to \infty} a_k^n(\lambda)
    &\geq C \lambda
    (1+\lambda)^{-\sum_{j=k}^{\infty} (1-\alpha_j)}.
  \end{align*}

  Now fix a $k$ such that $\sum_{i \geq k} (1-\alpha_i) < 1$. 
  From the inequality above, we may conclude that
  $\lim_{\lambda \to \infty} \liminf_{n \to\infty} a_{k}^n(\lambda) = \infty$.
  
  Using Proposition \ref{prop:ppp-invariance}, and taking an arbitrary $M
  \in \sN$, we can write:
  \begin{align*}
    Z_{x|_{k-1}}^n &= \sum_{x_k} \frac{\gamma_k(x|_k)
      Z_{x|_k}^n(\lambda)}{1+\gamma_k(x|_k) Z_{x|_k}^n(\lambda)}\\
    &\substack{\DD\\=} \sum_{x_k} \frac{\gamma_k(x|_k)
      (a_k^n(\lambda))^{1/\alpha_k}}{1+\gamma_k(x|_k)
      (a_k^n(\lambda))^{1/\alpha_k}}\\
    &\geq  \sum_{x_k = 1}^M \frac{\gamma_k(x|_k)
      (a_k^n(\lambda))^{1/\alpha_k}}{1+\gamma_k(x|_k)
      (a_k^n(\lambda))^{1/\alpha_k}}
    \xrightarrow[n \to \infty, \lambda \to \infty]{\emph{a.s.}} M.
  \end{align*}

  Since $M$ is arbitrary, it follows that
  $\lim_{\lambda \to \infty} \liminf_{n \to \infty}
  Z_{x|_{k-1}}^n(\lambda) = \infty$ in probability.  Knowing that
  $Z_{x|_{k-2}}^n = \sum_{x_{k-1}} h_{x|_{k-1}} (Z_{x|_{k-1}}^n)$, we
  can use analogous arguments to show that $\lim_{\lambda \to \infty}
  \liminf_{n \to \infty} Z_{x|_{k-2}}^n = \infty$ in
  probability. Iterating we conclude that:
    \begin{displaymath}
      \lim_{\lambda \to \infty} \liminf_{n \to \infty} Z_\emptyset^n
      (\lambda) = \infty \text{ in probability.}
      \qedhere
    \end{displaymath}
\end{proof}



%
\section{Convergence}
\label{sec:conv-k-proc}

A natural question to ask after constructing the K process on a tree with
infinite depth ($\sY$) is whether this process is the limit, in some
sense, of the K processes on trees with finite depth ($X_k$) as the
depth grows to infinity. We will address this question in this section
and prove Theorem \ref{thm:conv-infinite-process} stated in Section
\ref{sec:constr-main-result}.

\begin{lema}
  \label{lema:w-control}
  Suppose that $\frac{1-\alpha_{k+1}}{1-\alpha_k} \to 0$ as $k \to
  \infty$. Then:
  \begin{equation}
    \label{eq:w-control}
    \lim_{k \to \infty} \sE \left[ \left| W(x|_k) - 1 \right|^{\alpha_k} \right]
    = 0
  \end{equation}
\end{lema}

We note that the condition of Lemma~\ref{lema:w-control} holds under the 
non-triviality condition~\eqref{eq:nao-trivial}.

\begin{proof}
  We know the Laplace transform of $W(x|_k)$ from Proposition
  \ref{prop:def-W}. With it we can apply Proposition 1.1.12 from
  \cite{kn:st} to obtain the characteristic function of
  $W(x|_k) - 1$:
  \begin{align}
    \label{eq:w-control-1}
    \varphi_k (u) &:= \sE \left[ e^{i u (W(x|_k) - 1)} \right]\\
    \notag
    &= \exp\left\{
      - |u|^{\alpha_{k+1}} \left[
        \cos\left(\frac{\pi \alpha_{k+1}}{2} \right)
        -i \sgn(u) \sin\left( \frac{\pi \alpha_{k+1}}{2} \right)
      \right] - i u
    \right\}\\
    \notag
    &= \exp\left\{
      - |u|^{\alpha_{k+1}} \cos\left(\frac{\pi \alpha_{k+1}}{2} \right)
        -i \left[ u - |u|^{\alpha_{k+1}} \sgn(u) \sin\left( \frac{\pi \alpha_{k+1}}{2} \right)
      \right]
    \right\}.
  \end{align}

  Theorem 2.2 from \cite{kn:l} states that:
  \begin{align*}
    \sE \left[
       \left| W(x|_k) - 1 \right|^{\alpha_k}
     \right] &=
     \frac{1}{\cos\left( \frac{\pi \alpha_k}{2} \right)} \real \left[
       \frac{\alpha_k}{\Gamma(1-\alpha_k)} \int_0^\infty \frac{1 -
         \varphi_k(-u)}{u^{1+\alpha_k}} d u
     \right].
  \end{align*}

  Note that $ \frac{\alpha_k}{\cos\left( \frac{\pi \alpha_k}{2}\right)
    \Gamma(1-\alpha_k)}$ converges to $\frac{2}{\pi}$ as $k \to
  \infty$. So we are left with showing that the real part of the integral
  converges to zero. Fixing an arbitrary $\epsilon > 0$ we can write:
  \begin{align}
    \notag
    \real& \left[
      \int_0^\infty \frac{1 -
        \varphi_k(-u)}{u^{1+\alpha_k}} d u
     \right]\\
     \notag
     &= \int_0^\infty \frac{1}{u^{1+\alpha_{k}}}
     \left[ 1 - \exp\left\{
         -u^{\alpha_{k+1}} \cos\left(\frac{\pi \alpha_{k+1}}{2}
         \right)
       \right\} 
       \cos\left( u -
         u^{\alpha_{k+1}} \sin \left(\frac{\pi \alpha_{k+1}}{2}
         \right)
       \right)
     \right] d u \\
     \label{eq:w-control-2}
     &= \int_0^\epsilon \frac{1}{u^{1+\alpha_{k}}}
     \left[ 1 - \exp\left\{
         -u^{\alpha_{k+1}} \cos\left(\frac{\pi \alpha_{k+1}}{2}
         \right)
       \right\}
       \cos\left( u -
         u^{\alpha_{k+1}} \sin \left(\frac{\pi \alpha_{k+1}}{2}
         \right)
       \right)
     \right] d u \\
     \label{eq:w-control-3}
     &+ \int_\epsilon^\infty \frac{1}{u^{1+\alpha_{k}}}
     \left[ 1 - \exp\left\{
         -u^{\alpha_{k+1}} \cos\left(\frac{\pi \alpha_{k+1}}{2}
         \right)
       \right\}
       \cos\left( u -
         u^{\alpha_{k+1}} \sin \left(\frac{\pi \alpha_{k+1}}{2}
         \right)
       \right)
     \right] d u.
  \end{align}

  To control \eqref{eq:w-control-3}, note that the 
  integrand converges to zero as $k \to \infty$ and can be bounded by
  $2/u^{1+\alpha_k} \leq 2/u^{3/2}$ for big enough $k$. Therefore, by
  dominated convergence, \eqref{eq:w-control-3} converges
  to zero as $k \to \infty$ for any choice of $\epsilon > 0$.

  We control \eqref{eq:w-control-2}, using $1-e^{-x} \leq x$, as follows. 
  It is bounded above by
  \begin{align}\notag
    &\int_0^\epsilon \frac{1}{u^{1+\alpha_{k}}}
     \left[
       u^{\alpha_{k+1}} \cos\left(\frac{\pi \alpha_{k+1}}{2} \right)
       - \log \cos\left( u -
         u^{\alpha_{k+1}} \sin \left(\frac{\pi \alpha_{k+1}}{2}
         \right)
       \right)
     \right] d u\\
     \notag
     &=
     \int_0^\epsilon \frac{1}{u^{1+\alpha_{k}}}
     u^{\alpha_{k+1}} \cos\left(\frac{\pi \alpha_{k+1}}{2} \right)
     d u
     - \int_0^\epsilon \frac{1}{u^{1+\alpha_{k}}}
     \log \cos\left( u -
       u^{\alpha_{k+1}} \sin \left(\frac{\pi \alpha_{k+1}}{2}
       \right)
     \right) d u\\
     \label{eq:w-control-4}
     &= \frac{\epsilon^{\alpha_{k+1} - \alpha_k}}{\alpha_{k+1} -
       \alpha_k} \cos\left(\frac{\pi \alpha_{k+1}}{2} \right)
     - \int_0^\epsilon  \frac{1}{u^{1+\alpha_{k}}}
     \log \cos\left( u -
       u^{\alpha_{k+1}} \sin \left(\frac{\pi \alpha_{k+1}}{2}
       \right)
     \right) d u.
  \end{align}

  Since $\frac{1-\alpha_{k+1}}{1-\alpha_k} \xrightarrow{k \to \infty}
  0$ by hypothesis and $\cos(\pi x /2)/(1-x) \xrightarrow{x \to 1}
  \frac{\pi}{2}$, we can rewrite the leftmost term from
  \eqref{eq:w-control-4} as:
  \begin{align*}
    \frac{\epsilon^{\alpha_{k+1} - \alpha_k}}{\alpha_{k+1} -
      \alpha_k}& \cos\left(\frac{\pi \alpha_{k+1}}{2} \right)
    = \epsilon^{\alpha_{k+1} - \alpha_k}
    \frac{1- \alpha_{k+1}}{\alpha_{k+1} - \alpha_k}
    \frac{\cos\left(\frac{\pi \alpha_{k+1}}{2}
      \right)}{1-\alpha_{k+1}}\\
    &= \epsilon^{\alpha_{k+1} - \alpha_k}
    \left(
      \frac{1 - \alpha_k}{1- \alpha_{k+1}} - 1
    \right)^{-1}
    \frac{\cos\left(\frac{\pi \alpha_{k+1}}{2}
      \right)}{1-\alpha_{k+1}}
    \xrightarrow{k \to \infty} 0.
  \end{align*}

  To control the integral in \eqref{eq:w-control-4}, let us first remark that
  there exists an $\epsilon_0 > 0$ such that if $|x| < \epsilon_0$ then
  $-\log \cos x < x^2$. 
  With this, for small enough $\epsilon > 0$ we can write:
  \begin{align*}
    - \int_0^\epsilon& \frac{1}{u^{1+\alpha_{k}}}
    \log \cos\left( u -
      u^{\alpha_{k+1}} \sin \left(\frac{\pi \alpha_{k+1}}{2}
      \right)
    \right) d u\\
    &\leq \int_0^\epsilon \frac{1}{u^{1+\alpha_{k}}}
    \left( u -
      u^{\alpha_{k+1}} \sin \left(
        \frac{\pi \alpha_{k+1}}{2}
      \right)
    \right)^2 d u\\
    &= \int_0^\epsilon
    u^{1-\alpha_k}
    - 2 u^{\alpha_{k+1} - \alpha_k}  \sin\left( \frac{\pi \alpha_{k+1}}{2} \right)
    + u^{2 \alpha_{k+1} - \alpha_k - 1}  \sin^2\left( \frac{\pi \alpha_{k+1}}{2} \right)
    d u\\
    &=
    \frac{\epsilon^{2-\alpha_k}}{2-\alpha_k}
    - 2 \frac{\epsilon^{\alpha_{k+1} - \alpha_k + 1}}{\alpha_{k+1} -
      \alpha_k + 1}  \sin\left( \frac{\pi \alpha_{k+1}}{2} \right)
    + \frac{\epsilon^{2 \alpha_{k+1} - \alpha_k}}{2\alpha_{k+1} -
      \alpha_k}  \sin^2\left( \frac{\pi \alpha_{k+1}}{2} \right)\\
    &\xrightarrow{k \to \infty} 0.
    \qedhere
  \end{align*}
\end{proof}

\begin{proposicao}[Finite dimensional convergence]
  \label{prop:clock-fin-dim}
  If $\sum_k (1-\alpha_k) < \infty$, then for every fixed $t > 0$ and
  $k \in \sN_*$, $\theta_k^n(t)$ converges in probability to
  $\theta_k^\infty(t)$ as $n \to \infty$.
\end{proposicao}

\begin{proof}
  We will assume that $\frac{1-\alpha_{k+1}}{1-\alpha_k} \to 0$ as $k
  \to \infty$. When this is not true then the conditions of Theorem
  \ref{thm:trivial} hold and we can use an analogous argument to show
  that $\theta_k^n(t) \xrightarrow{n\to\infty} 0$ in probability for
  every $t > 0$.

  Let us first show the case $k = 1$.  Taking $Z_{x|_n} = |W(x|_n) -
  1|$ in the proof of Proposition
  \ref{prop:laplace-clock} we obtain:
  \begin{align}
    \notag
    &\sE \left[
      \exp\left\{
        - |\theta_1^n(t) - \widetilde{\theta}_1^n(t)|
      \right\}
    \right] \geq
    \sE \left[
      \exp\left\{
        - \sum_{x|_n} L(x|_n, \theta_1^n(t))
        \left|1 - W(x|_n)\right|
      \right\}
    \right]\\
    \label{eq:clock-fin-dim-1}
    &=
    \sE \left[ \exp \left\{
        -t \sum_{x_1} h_{x|_1} \left(
          \sum_{x_2} h_{x|_2} \left(
            \cdots
            \sum_{x_n} h_{x|_n} \left( |W(x|_n) - 1| \right)
            \cdots
          \right)
        \right)
      \right\} \right].
  \end{align}

  Now proceeding as in Theorem \ref{thm:trivial}, let us define:
  \begin{displaymath}
    Z_{x|_k}^n := \begin{cases}
      \sum_{x_{k+1}} h_{x|_{k+1}} \left( Z_{x|_{k+1}}^n \right),
        & \!\text{ if } k < n;\\
        |W(x|_n) - 1|, & \!\text{ if } k = n
    \end{cases};
    \quad
    a_k^n := \begin{cases}
      \sE \left[
        Z_{x|_k}^n
      \right],
      & \!\text{ if } k < n;\\
      \sE \left[
        |W(x|_n) - 1|^{\alpha_n}
      \right], & \!\text{ if } k = n.
    \end{cases}
  \end{displaymath}

  Lemma \ref{lema:w-control} states that $a_n^n \to 0$ as $n \to
  \infty$.  Following the proof of Theorem \ref{thm:trivial}, we get
  \eqref{eq:trivial-1} and the first equality of
  \eqref{eq:trivial-1-2}, from which we obtain that $a_{k-1}^n \leq
  a_k^n$ for every $k \leq n$. Therefore $a_k^n \leq a_n^n
  \xrightarrow{n \to \infty} 0$ for every $k < n$ and $Z_\emptyset^n
  \to 0$ in the $L_1$ norm as $n \to \infty$.  By applying the Dominated
  Convergence Theorem on~\eqref{eq:clock-fin-dim-1}, we conclude the
  case $k = 1$.

  For the general case, using Remark \ref{obs:pedacos-clock} and the
  previous result, we can write:
  \begin{align}
    \notag
    \sE \left[
      e^{-|\theta_{k+1}^n(t) - \widetilde{\theta}_{k+1}^n(t)|}
    \right] &\geq
    \sE \left[
      \exp\left\{
        - \sum_{x|_k} \left|
          \theta_{x|_k}^n(L(x|_k, t)) - \widetilde{\theta}_{x|_k}^n(L(x|_k, t))
      \right|
      \right\}
    \right]\\
    \label{eq:clock-fin-dim-2}
    &=
    \sE \left[
      \exp\left\{
        - \sum_{x|_k} L(x|_k, t) Z_{x|_k}^n
      \right\}
    \right].
  \end{align}

  Note that, since $L(x|_k, t)$ is independent from $Z_{x|_k}^n$, then:
  \begin{align*}
    \sE \left[
      \sum_{x|_k} L(x|_k, t)  Z_{x|_k}^n
    \right]
    =
    \sum_{x|_k}
    \sE \left[
         L(x|_k, t)  Z_{x|_k}^n
      \right]
    = t a_k^n \xrightarrow{n \to \infty} 0.
  \end{align*}

  We conclude the proof by applying the Dominated Convergence Theorem on
  \eqref{eq:clock-fin-dim-2}.
\end{proof}

\begin{lema}
  \label{lema:sup-sum-gamma}
  Almost surely:
  \begin{equation}
    \label{eq:sup-sum-gamma}
    \sup_{n} \sum_{x|_n} \widebar{\gamma}_n(x|_n) < \infty.
  \end{equation}
\end{lema}
\begin{proof}
  Let $A_n := \{ x|_n \in \sN_*^n: \widebar{\gamma}_n(x|_n) > 1 \}$
  and $m_n := \max\{ \widebar{\gamma}_n(x|_n): x|_n \in
  \sN_*^n\}$. Using Proposition \ref{prop:def-W}, we know that almost
  surely:
  \begin{align*}
    W(\emptyset) &= \lim_{n \to \infty} \sum_{x|_n \in \sN_*^n} \left(
      \widebar{\gamma}_n(x|_n) \right)^{\alpha_{n+1}}
    \geq \limsup_{n \to \infty} \sum_{x|_n \in A_n} \left(
      \widebar{\gamma}_n(x|_n) \right)^{\alpha_{n+1}}
    \geq \limsup_{n \to \infty} |A_n|;\\
    W(\emptyset) &= \lim_{n \to \infty} \sum_{x|_n \in \sN_*^n} \left(
      \widebar{\gamma}_n(x|_n) \right)^{\alpha_{n+1}}
    \geq \limsup_{n \to \infty} \left( m_n \right)^{\alpha_{n+1}}
    = \limsup_{n \to \infty} m_n.
  \end{align*}

  Since $W(\emptyset) < \infty$, then $\limsup_n |A_n| < \infty$ and
  $\limsup_n m_n < \infty$ \emph{a.s.}. With this we can write:
  \begin{align*}
    \sum_{x|_n} \widebar{\gamma}_n(x|_n)
    &=
    \sum_{x|_n \in A_n} \widebar{\gamma}_n(x|_n) +
    \sum_{x|_n\ \not\in A_n} \widebar{\gamma}_n(x|_n)\\
    &\leq m_n |A_n| +
    \sum_{x|_n\ \not\in A_n} \left( \widebar{\gamma}_n(x|_n) \right)^{\alpha_{n+1}}.
  \end{align*}

  Because of the previous remark, the first term of this last sum is
  bounded \emph{a.s.}~by a constant, while the $\limsup$ of the second
  term is dominated by $W(\emptyset)$. This concludes this proof.
\end{proof}

\begin{lema}
  \label{lema:theta-subordinator}
  For almost every $\underline\gamma$ we have that
  $\theta_1^n$ is a subordinator for every $n \in
  \sN_*$. Furthermore:
  \begin{displaymath}
    \sE_\gamma \left[ \theta_1^n(t) \right] = t \sum_{x|_n} \overline{\gamma}_n(x|_n).
  \end{displaymath}
\end{lema}

\begin{proof}
  The expected value is computed in the proof of Lemma 4.5 from
  \cite{kn:fgg}. Let us prove that $\theta_1^n$ is a subordinator.  We
  will do this by induction on $n$. The case $n = 1$ is a direct
  consequence of the independent and stationary increments of a
  Poisson point process and the fact the variables used to construct
  $\theta_1^1 = \clock_1$ are independent.

  Assuming $n \geq 2$, note that for $t, s > 0$:
  \begin{align*}
    \theta_1^n(t+s) - \theta_1^n(t) &=
    \sum_{x} \sum_{i = N^{n, x}(\theta_1^{n-1}(t)) + 1}^{N^{n,
        x}(\theta_1^{n-1}(t+s))}
    \gamma_n(X_{n-1}(\sigma_i^{n, x}) x) T^{n, x}_i.
  \end{align*}

  Fix $0 = t_0 <  t_1 < \ldots < t_k$ and let us look at the joint
  distribution of $(\theta_1^n(t_i) - \theta_1^{n}(t_{i-1}))_{i = 1,
    \ldots, k}$.

  $\theta_1^n(t_i) - \theta_1^{n}(t_{i-1})$, for varying values of
  $i$, depends on disjoint intervals of the Poisson processes $N^{n,
    x}$. Each one with length $\theta_1^{n-1}(t_{i}) -
  \theta_1^{n-1}(t_{i-1})$. Because of the induction hypothesis these
  lengths are independent and each one has the same law as
  $\theta_1^{n-1}(t_i - t_{i-1})$.

  Note that, as can be justified by Remark \ref{obs:infinitos-niveis}, 
  we have that for  
  fixed $t > 0$, $X_{n-1}(\theta_1^{n-1}(t)) = (\infty, \ldots, \infty)$
  \emph{a.s.}. The instant $\theta_1^{n-1}(t)$ is thus a renewal time for
  $X_{n-1}$, 
  and, by construction, the law of $X_{n-1}$ right
  after such a renewal is the same as the law right after the instant
  zero.

  $\theta_1^n(t_i) - \theta_1^n(t_{i-1})$ also depends on the values
  of $X_{n-1}(r)$ for $r \in (\theta_1^{n-1}(t_{i-1}),
  \theta_1^{n-1}(t_i))$. Both ends of this interval are
  renewals, so what happens to $X_{n-1}$ inside such an interval is 
  independent of what happens on other such intervals (namely,
  $(\theta_1^{n-1}(t_{j-1}),\theta_1^{n-1}(t_j))$, $j\ne i$).

  Therefore we have that $\theta_1^n(t_i) - \theta_1^n(t_{i-1})$, $i\geq1$,
  are independent, and the law of each $\theta_1^n(t_i) - \theta_1^n(t_{i-1})$
  depends only on $t_i - t_{i-1}$. That is, $\theta_1^n$ has independent and
  stationary increments.
\end{proof}

\begin{observacao}
  \label{obs:theta-inf-subordinador}
  If the random environment  $\{\gamma_k(x|_k): k \in \sN_*, x|_k \in
  \sN_*^k\}$ is fixed, then Proposition \ref{prop:clock-fin-dim} and
  Lemma \ref{lema:theta-subordinator} readily imply that
  $\theta_1^\infty$ is also a subordinator.
\end{observacao}

\begin{teorema}
  \label{thm:clock-conv}
  Assuming that $\sum_{k}(1-\alpha_k) < \infty$, then
  for every $k \in \sN_*$, $\theta_k^n - \theta_k^\infty$ converges
  weakly to the identically null  function in the Skorohod topology as
  $n \to \infty$.
\end{teorema}

\begin{observacao}
  This is a stronger result than simply stating that $\theta_k^n$
  converges weakly to $\theta_k^\infty$ as $n \to \infty$ in the
  Skorohod topology. Convergence in the Skorohod topology to a
  continuous function is equivalent to uniform convergence on
  compact sets.
\end{observacao}

\begin{proof}[Proof of Theorem \ref{thm:clock-conv}]
  Under the triviality condition \eqref{eq:trivial}, this theorem is a
  direct corollary of Proposition \ref{prop:clock-fin-dim}, together
  with the observation that each $\theta_k^n$ is an
  increasing function. From now on let us assume the non-triviality
  condition \eqref{eq:nao-trivial}.

  Let us start with the case $k = 1$. Fix $\underline\gamma$ and
  let us show the convergence for almost every such choice.

  Since we have shown the finite dimensional convergence in Proposition
  \ref{prop:clock-fin-dim}, then Theorem 7.8 from Chapter 3 of
  \cite{kn:ek} says that if $\{|\theta_1^n - \theta_1^\infty|\}$
  is relatively compact, then it is true that it converges weakly to
  the identically null function.

  Using part (b) from Theorem 8.6 from Chapter 3 of \cite{kn:ek}, to
  show relative compactness it is enough to show that for and $0 < s <
  \delta$ and $t > 0$:
  \begin{align}
    \label{eq:clock-conv-2}
    \sE_\gamma \left[
      \left|
        \left(\theta_1^n(t+s) - \theta_1^\infty(t+s)  \right)
        -
        \left(\theta_1^n(t) - \theta_1^\infty(t)  \right)
      \right|
      \middle| \HH_t^n
    \right] &\leq
    2 \delta \sup_{m \in \sN} \sum_{x|_m} \widebar{\gamma}_m(x|_m)
    \xrightarrow[\textit{a.s.}]{\delta \to 0}  0,
  \end{align}
  where $\HH_t^n$ is the $\sigma$-algebra generated by
  $\{\theta_1^n(r) - \theta_1^\infty(r): r \leq t\}$.

  The almost sure convergence is a direct consequence of Lemma
  \ref{lema:sup-sum-gamma}. To prove the inequality, let $\HH_t$ be the
  $\sigma$-algebra generated by the random variables $\{\theta_1^n(r):
  r \leq t, n \in \sN_* \}$. Since $\HH_t^n \subseteq \HH_t$, we can write:
  \begin{align}
    \notag
    \sE_\gamma &\left[
      \left|
        \left(\theta_1^n(t+s) - \theta_1^\infty(t+s)  \right)
        -
        \left(\theta_1^n(t) - \theta_1^\infty(t)  \right)
      \right|
      \middle| \HH_t^n
    \right]\\
    \notag
    &\leq
    \sE_\gamma \left[
      \theta_1^n(t+s) - \theta_1^n(t)
      \middle| \HH_t^n
    \right]
    +
    \sE_\gamma \left[
      \theta_1^\infty(t+s) - \theta_1^\infty(t)
      \middle| \HH_t^n
    \right]\\
    \label{eq:clock-conv-1}
    &\leq
    \sE_\gamma \left[
      \sE_\gamma \left[
        \theta_1^n(t+s) - \theta_1^n(t)
        \middle| \HH_t
      \right]
      \middle| \HH_t^n
    \right]
    +
    \sE_\gamma \left[
      \sE_\gamma \left[
        \theta_1^\infty(t+s) - \theta_1^\infty(t)
        \middle| \HH_t
      \right]
      \middle| \HH_t^n
    \right].
  \end{align}

  Note that $\theta_1^n(t+s) - \theta_1^n(t)$ only depends on the
  values of $\theta_1^m(r)$ for $m \geq n$ and $r \leq t$ through the
  values of $\theta_1^n(r)$, $r \leq t$. So we can use
  Lemma \ref{lema:theta-subordinator} to compute the value of the first
  term:
  \begin{align*}
    \sE_\gamma \left[
      \theta_1^n(t+s) - \theta_1^n(t)
      \middle| \HH_t
    \right] =
    \sE_\gamma\left[ \theta_1^n(s) \right]
    = s \sum_{x|_n} \overline{\gamma}_n(x|_n).
  \end{align*}

  Since $\theta_{1}^n(t)$ and $\theta_1^n(t+s)$ converges in
  probability to $\theta_1^\infty(t)$ and $\theta_1^\infty(t+s)$
  respectively (Proposition \ref{prop:clock-fin-dim}), we can
  take an increasing sequence $n_m$ such that the convergence is
  almost sure. Then using Fatou's Lemma, we can conclude:
  \begin{align*}
    \sE_\gamma \left[
      \theta_1^\infty(t+s) - \theta_1^\infty(t)
      \middle| \HH_t
    \right]
    \leq
    \liminf_{m \to \infty}
    \sE_\gamma \left[
      \theta_1^{n_m}(t+s) - \theta_1^{n_m}(t)
      \middle| \HH_t
    \right]
    \leq
    \liminf_{m \to \infty} s \sum_{x|_{n_m}} \widebar{\gamma}_{n_m}(x|_{n_m}).
  \end{align*}

  Using the last results on \eqref{eq:clock-conv-1} gives us
  \eqref{eq:clock-conv-2}. And completes the proof that $\theta_1^n$
  converges in probability to $\theta_1^\infty$ uniformly in compacts.

  $\theta_{k+1}^n - \theta_{k+1}^\infty$ converges in
  probability to the null function in the Skorohod topology if and
  only if for every $T > 0$:
  \begin{align}
    \label{eq:clock-conv-3}
    \sup_{0 \leq t \leq T} |\theta_{k+1}^{n}(t) -
    \theta_{k+1}^\infty(t)|
    \xrightarrow[n\to \infty]{P} 0.
  \end{align}

  For each $x|_k \in \sN_*^k$ fixed, define $\theta_{x|_k}^\infty$
  from $\theta_{k+1}^\infty$ in an analogous way as done in Definition
  \ref{def:pedacos-clock}. Using Remark \ref{obs:pedacos-clock} and
  the last case, we can conclude that $\theta_{x|_k}^n -
  \theta_{x|_k}^\infty$ converges to the null function in probability
  in the Skorohod topology as $n \to \infty$.

  Take $V_1 \subset V_2 \subset \ldots \sN_*^k$ any increasing
  sequence of sets from $N_*^k$ such that $\bigcup_{l=1}^\infty V_l =
  \sN_*^k$ and each $V_l$ is finite. For an arbitrary $\epsilon > 0$,
  fix an $l$ such that 
  $$\sP_\gamma \left( \sum_{x|_k \not \in V_l}
  \theta_{x|_k}^\infty(L(x|_k, T)) > \epsilon \right) < \epsilon.$$
  Thence,
  \begin{align*}
    &\sup_{0 \leq t \leq T} |\theta_{k+1}^{n}(t) -
    \theta_{k+1}^\infty(t)|
    =
    \sup_{0 \leq t \leq T} \left|
      \sum_{x|_k} \theta_{x|_k}^{n}(L(x|_k, t)) -
      \theta_{x|_k}^\infty(L(x|_k, t))
    \right|\\
    &\leq
    \sup_{0 \leq t \leq T} \sum_{x|_k \in V_l} \left|
       \theta_{x|_k}^{n}(L(x|_k, t)) -  \theta_{x|_k}^\infty(L(x|_k,
       t))
    \right|
    + \sum_{x|_k \not\in V_l}
    \theta_{x|_k}^{n}(L(x|_k, T))\\
    &\mbox{ }\hspace{1cm}+ \sum_{x|_k \not\in V_l}
    \theta_{x|_k}^\infty(L(x|_k, T))\\
    &\leq
    \sup_{0 \leq t \leq T} \sum_{x|_k \in V_l} \left|
       \theta_{x|_k}^{n}(t) -  \theta_{x|_k}^\infty(t)
    \right|
    + \sum_{x|_k \not\in V_l}
    \theta_{x|_k}^{n}(L(x|_k, T))
    + \sum_{x|_k \not\in V_l}
    \theta_{x|_k}^\infty(L(x|_k, T)).
  \end{align*}

  The first term converges to $0$ in probability as $n \to \infty$
  because $V_l$ is finite. The third term is controlled by our choice
  of $\epsilon$. Using an analogous argument as in
  \eqref{eq:clock-fin-dim-2}, we can show that the second term
  converges to the third in probability as $n \to \infty$.
\end{proof}

\begin{observacao}
  \label{obs:clock-modif-conv}
  Using analogous arguments, we can prove Theorem \ref{thm:clock-conv}
  with $\widetilde\theta_k^n$ in the place of $\theta_k^n$. The only
  significant change in the proof comes from the equality $\sE_\gamma
  [\widetilde\theta_1^n(t)] = t W(\emptyset)$. However this actually
  slightly simplifies the proof.
\end{observacao}

\begin{corolario}
  \label{cor:clock-composition}
  If $\sum_i(1-\alpha_i) < \infty$, then for every $j, k \in \sN_*$
  with $j < k$ \emph{a.s.}:
  \begin{align*}
    \theta_j^\infty = \theta_{k}^\infty \circ \theta_j^{k-1}.
  \end{align*}
\end{corolario}

\begin{proof}
  Since, by Theorem \ref{thm:clock-conv},  $\theta_j^{n}$ and
  $\theta_{k}^n$ converge in probability in the uniform norm to
  $\theta_j^\infty$ and $\theta_k^\infty$ respectively, we can
  take an increasing sequence $n_m$ such that this convergence is
  almost sure. Fixing such a sequence, we can write that for any $T > 0$:
  \begin{align*}
   \sup_{t \in [0, T]} &\left|
      \theta_j^\infty(t) - \theta_{k}^\infty \left( \theta_j^{k-1}(t) \right)
    \right|\\ \leq
    \sup_{t \in [0, T]} &\left|
      \theta_j^\infty(t) - \theta_j^{n_m}(t)
    \right| +
    \sup_{t \in [0, T]} \left|
     \theta_{k}^{n_m} \left( \theta_j^{k-1}(t) \right) -
     \theta_{k}^{\infty} \left( \theta_j^{k-1}(t) \right)
    \right|\\
    \xrightarrow[\textit{a.s.}]{m \to \infty}& 0.
    \qedhere
  \end{align*}
\end{proof}

\begin{corolario}
  \label{cor:discontinuities}
  Suppose that $\sum_i (1 - \alpha_i) < \infty$. For any fixed $k \in
  \sN_*$, if $s \geq 0$ is a discontinuity point of $\theta_k^\infty$,
  then $s = \sigma_i^{k, x}$ for some $k, x \in \sN_*$.
\end{corolario}
\begin{proof}
  Using Theorem \ref{thm:clock-conv}, we can take an increasing
  sequence $(n_m)_m$ such that, for an $T > s$ fixed arbitrarily:
  \begin{displaymath}
    \sup_{t \in [0, T]} | \theta_k^{n_m}(t) - \theta_k^{\infty}(t) |
    \xrightarrow[a.s.]{m \to \infty} 0.
  \end{displaymath}

  Since $s$ is a discontinuity point of $\theta_k^\infty$ and this is
  an non-decreasing càdlàg function, there exists an $\epsilon >
  0$ such that $\theta_k^\infty(s) - \theta_k^\infty(s-) > \epsilon$,
  and this implies that $\theta_k^\infty(s) - \theta_k^\infty(s-h) >
  \epsilon$ for every $h \in (0, s)$.

  For an arbitrary $h \in (0, s)$ we can write:
  \begin{align*}
    \theta_k^{n_m}(s) - \theta_k^{n_m}(s-h)
    &= \theta_k^{n_m}(s) - \theta_k^\infty(s)\\
    &+ \theta_k^{\infty}(s) - \theta_k^\infty(s-h)\\
    &+ \theta_k^\infty(s-h) - \theta_k^{n_m}(s-h).
  \end{align*}

  The first and third terms of this last equation converges
  \emph{a.s.}  to zero uniformly in $h$ as $m \to \infty$, while the
  second term is always greater than $\epsilon$. Therefore we conclude
  that, $\inf_{h \in (0, s)} \theta_k^{n_m}(s) - \theta_k^{n_m}(s-h) >
  \epsilon/2$ for large enough $m$, so $s$ is a discontinuity point of
  $\theta_k^{n_m}$. Finally Remark \ref{obs:infinitos-niveis} states
  that the set of discontinuity points of $\theta_k^{n_m}$ is
  ${\{\sigma_i^{k, x}: i, x \in \sN_* \}}$.
\end{proof}

\begin{corolario}
  \label{cor:tempo-infinito}
  Under the non-triviality assumption \eqref{eq:nao-trivial}, for any
  $k \in \sN_*$:
  \begin{displaymath}
    \theta_k^\infty(\sR) = \{t \geq 0: Y_k(t) = \infty\} \textit{ a.s.}
  \end{displaymath}
\end{corolario}

\begin{proof}
  Start by taking $t \in \theta_k^\infty(\sR)$ and an $s \geq 0$ such
  that $\theta_k^\infty(s) = t$. Let us assume for contradiction
  hypothesis that $Y_k(t) = x < \infty$.

  By definition, since $Y_k(t) = x$, there exists an $i \in
  \sN_*$ such that $\theta_k^\infty (\sigma_i^{k, x}-) \leq t <
  \theta_k^\infty(\sigma_i^{k, x})$.

  Since $t = \theta_k^\infty(s)$ and $\theta_k^\infty$ is strictly
  increasing (Theorem \ref{thm:nao-trivial}), the right
  inequality implies that $s < \sigma_i^{k, x}$, which in turn implies
  that $t = \theta_k^\infty(s) < \theta_k^{\infty}(\sigma_i^{k, x}-)$,
  which contradicts the first inequality from the last
  paragraph. Therefore $Y_k(t) = \infty$.

  Let us assume that $t > 0$ is such that $Y_k(t) =
  \infty$. Take $s = \inf \left\{ r > 0: \theta_k^\infty(r) > t
  \right\}$. Using the right continuity of $\theta_k^\infty$, we
  conclude that $\theta_k^\infty(s) \geq t$. Let us now assume, as a
  contradiction hypothesis, that $\theta_k^\infty(s) > t$. By the
  definition of $s$, we know that $\theta_k^\infty(s-) \leq
  t$. Therefore $\theta_k^\infty(s-) \neq \theta_k^\infty(s)$.

  Using Corollary \ref{cor:discontinuities}, we obtain that $s =
  \sigma_i^{k, x}$ for some $i, x \in \sN_*$. Therefore
  $\theta_k^\infty(\sigma_i^{k, x}-) \leq t <
  \theta_k^\infty(\sigma_i^{k, x})$, which implies that $Y_k(t) = x <
  \infty$, contradicting our choice for $t$.
\end{proof}

\begin{observacao}
  If $\sum_i(1-\alpha_i) < \infty$, Corollary
  \ref{cor:clock-composition} and \ref{cor:tempo-infinito} readily
  imply that:
  \begin{displaymath}
    \{t \geq 0: Y_k(t) = \infty \} \subseteq \{t \geq 0: Y_{k+1}(t) =
    \infty \}
  \end{displaymath}
  almost surely for any $k \in \sN_*$.
\end{observacao}

\begin{corolario}
  \label{cor:lebesgue-tempo-infinito}
  Under the non-triviality condition \eqref{eq:nao-trivial}, the set
  $\{t \geq 0: Y_k(t) = \infty\}$ has null Lebesgue measure
  \emph{a.s.} for every $k \in \sN_*$.
\end{corolario}

\begin{proof}
  We will prove that the following equality is valid for every $t > 0$
  almost surely:
  \begin{equation}
    \label{eq:lebesgue-tempo-infinito-1}
    \theta_k^\infty(t) = \sum_{x = 1}^\infty \sum_{i=1}^{N^{k, x}(t)}
    \theta_k^\infty(\sigma_i^{k, x})-
    \theta_k^\infty(\sigma_i^{k, x} -).
  \end{equation}
  With this we are showing that $\theta_k^\infty$ is a step function,
  therefore its image has null Lebesgue measure.  The result then
  follows from Corollary \ref{cor:tempo-infinito}.

  Note that \eqref{eq:lebesgue-tempo-infinito-1} is not a direct
  consequence of the uniform convergence in Theorem
  \ref{thm:clock-conv}. It is possible to construct a sequence of step
  functions $(f_n)$ that converge uniformly to another function $f$,
  all having exactly the same discontinuities but $f$ itself is not a
  step function.

  For the rest of the proof, let us fix $\underline\gamma$ and prove 
  the result for almost every such choice.

  Since $\theta_k^\infty$ is a non-decreasing, càdlàg function,
  it is the distribution function of a measure. The right hand side of
  \eqref{eq:lebesgue-tempo-infinito-1} can be interpreted as the sum
  over some points of this measure. Therefore we conclude that:
  \begin{displaymath}
    \theta_k^\infty(t) \geq \sum_{x = 1}^\infty \sum_{i=1}^{N^{k, x}(t)}
    \theta_k^\infty(\sigma_i^{k, x})-
    \theta_k^\infty(\sigma_i^{k, x} -).
  \end{displaymath}

  To show the reverse inequality, using Remark
  \ref{obs:clock-modif-conv}, let us take an increasing sequence $(n_m)_m$
  such that this convergence is almost sure. We note that for any
  fixed $N \in \sN_*$:
  \begin{align}
    \notag
    \theta_k^\infty(t) &= \lim_{m \to \infty}
    \widetilde{\theta}_k^{n_m}(t)\\
    \notag
    &= \lim_{m \to \infty}
    \sum_{x = 1}^\infty \sum_{i=1}^{N^{k, x}(t)}
    \widetilde\theta_k^{n_m}(\sigma_i^{k, x})-
    \widetilde\theta_k^{n_m}(\sigma_i^{k, x} -)\\
    \label{eq:lebesgue-tempo-infinito-2}
    &\leq \limsup_{m \to \infty}
    \sum_{x = 1}^N \sum_{i=1}^{N^{k, x}(t)}
    \widetilde\theta_k^{n_m}(\sigma_i^{k, x})-
    \widetilde\theta_k^{n_m}(\sigma_i^{k, x} -)\\
    \label{eq:lebesgue-tempo-infinito-3}
    &+ \limsup_{m \to \infty}
    \sum_{x = N+1}^\infty \sum_{i=1}^{N^{k, x}(t)}
    \widetilde\theta_k^{n_m}(\sigma_i^{k, x})-
    \widetilde\theta_k^{n_m}(\sigma_i^{k, x} -).
  \end{align}

  We also note that \eqref{eq:lebesgue-tempo-infinito-2} is equal to:
  \begin{displaymath}
    \sum_{x = 1}^N \sum_{i=1}^{N^{k, x}(t)}
    \theta_k^\infty(\sigma_i^{k, x})-
    \theta_k^\infty(\sigma_i^{k, x} -) \xrightarrow[a.s.]{N \to
      \infty}
    \sum_{x = 1}^\infty \sum_{i=1}^{N^{k, x}(t)}
    \theta_k^\infty(\sigma_i^{k, x})-
    \theta_k^\infty(\sigma_i^{k, x} -).
  \end{displaymath}

  To complete the proof, we need to show that
  \eqref{eq:lebesgue-tempo-infinito-3} converges to zero in
  probability as $N \to \infty$.  By analogous arguments as those used in
  the proof of Theorem \ref{thm:limit-clock}, 
  $(\widetilde{\theta}_k^n(\sigma_i^{k, x}) -
  \widetilde{\theta}_k^n(\sigma_i^{k, x}-))_n$ is appropriately a
  martingale. Therefore the sequence that we are taking the $\limsup$ of
  in \eqref{eq:lebesgue-tempo-infinito-3} is a positive martingale.
  Using again a martingale convergence theorem, 
  we conclude that the $\limsup$ in that expression is
  in fact a limit.

  Denoting \eqref{eq:lebesgue-tempo-infinito-3} by $K_N$ and again letting
  $\GG_k$ be the $\sigma$-algebra generated by all dynamical information up to the
  level $k$, we can use Fatou's Lemma to obtain that:
 \begin{align}
    \notag
    \sE_\gamma[K_N] &\leq \liminf_{m \to \infty} \sE_\gamma \left[
      \sum_{x = N+1}^\infty \sum_{i=1}^{N^{k, x}(t)}
      \widetilde\theta_k^{n_m}(\sigma_i^{k, x})-
      \widetilde\theta_k^{n_m}(\sigma_i^{k, x} -)
    \right]\\
    \notag
    &= \liminf_{m \to \infty} \sum_{x = N+1}^\infty
    \sE_\gamma \left[
      \sum_{i=1}^{N^{k, x}(t)}
      \sE_\gamma\left[
        \widetilde\theta_k^{n_m}(\sigma_i^{k, x})-
        \widetilde\theta_k^{n_m}(\sigma_i^{k, x} -)
        \middle| \FF_{k-1}
      \right]
    \right]\\
    \notag
    &= \liminf_{m \to \infty} \sum_{x = N+1}^\infty \sE_\gamma \left[
      \sum_{i=1}^{N^{k, x}(t)} \gamma_k(X_{k-1}(\sigma_i^{k,
        x}) x) W(X_{k-1}(\sigma_i^{k, x}) x)
    \right]\\
    \label{eq:lebesgue-tempo-infinito-4}
    &= \sum_{x = N+1}^\infty \sE_\gamma \left[
      \sum_{i=1}^{N^{k, x}(t)} \gamma_k(X_{k-1}(\sigma_i^{k,
        x}) x) W(X_{k-1}(\sigma_i^{k, x}) x)
    \right]\\
    \notag
    &\leq \sE_\gamma \left[ \sum_{x =1}^\infty
      \sum_{i=1}^{N^{k, x}(t)} \gamma_k(X_{k-1}(\sigma_i^{k,
        x}) x) W(X_{k-1}(\sigma_i^{k, x}) x)
    \right]\\
    \notag
    &= \sE_\gamma \left[ \widetilde\clock_k(t) \right] < \infty.
  \end{align}

  So we proved that the sum in \eqref{eq:lebesgue-tempo-infinito-4} is
  convergent and therefore converges to zero as $N \to
  \infty$. Finally $K_N$ converges in $L_1$ to zero and therefore in
  probability as well.
\end{proof}


Finally we can prove the main result of this paper, stated on Section
\ref{sec:constr-main-result}.

\begin{proof}[Proof of Theorem \ref{thm:conv-infinite-process}]
  We will omit the proof that $\sY$ is a càdlàg process, since the
  argument is quite similar to the ones used to prove the convergence.

  Take any increasing sequence $(b_n)$ of natural numbers. We will show
  that this sequence has a subsequence $(k_n)_{n}$ over which
  $X_{k_n}$ converges in probability to $\sY$ as $n \to \infty$. This
  implies that $X_k$ converges in probability to $\sY$.

  Theorem \ref{thm:clock-conv} guarantees that for each $j$ there exists a
  subsequence $(a_n)$ of $(b_n)$ such that $\theta_j^{a_n} -
  \theta_j^\infty$ converges almost surely in the uniform norm. Using
  Cantor's diagonal method we can show that there exists $(k_n)$ a
  subsequence of $(a_n)$ such that $\theta_j^{k_n} -
  \theta_j^{\infty}$ converges almost surely for all $j$. Fix such a
  subsequence. It is along it that we will show the convergence.

  Fix a realization of the process, we will show that for almost all
  such realizations and for every $T > 0$, there exists a sequence
  $(\lambda_n)$ of functions such that $\lambda_n: [0, \infty) \to [0,
  \infty)$ is strictly increasing Lipschitz continuous,
  that satisfies:
  \begin{align}
    \label{eq:conv-infinite-process-1}
    \lim_{n \to \infty} \sup_{t \in [0, T]} \rho\left(
      X_{k_n}(t), \sY(\lambda_n(t))
    \right) &= 0,\\
    \label{eq:conv-infinite-process-2}
    \lim_{n \to \infty} \sup_{t \in [0, T]} \left|
      \lambda_n(t) - t
    \right| &= 0.
  \end{align}

  Theorem 5.3 from Chapter 3 of \cite{kn:ek} guarantees that
  showing this is equivalent to showing that $X_{k_n}$ converges to
  $\sY$ almost surely in the Skorohod topology.

  We will show that for every $\epsilon > 0$ there exists a sequence
  $(\lambda_n)$ such that \eqref{eq:conv-infinite-process-2} is
  satisfied and the quantity in \eqref{eq:conv-infinite-process-1} is
  smaller than $\epsilon$.  This will imply that exists a sequence
  $(\lambda_n)$ that satisfies \eqref{eq:conv-infinite-process-1} and
  \eqref{eq:conv-infinite-process-2}.

  For a fixed $\epsilon > 0$, take $N, M \in \sN_*$ such that $\sum_{j
    \geq N} \frac{1}{2^j} < \epsilon/2$ and $1/(M+1) < \epsilon/2$, and
  assume that $n$ is large enough so that $k_n > N$.

  To construct $\lambda_n$, for $k_n > N$, take $S_n := \{
  \theta_j^{N}(\sigma_i^{j, x}-), \theta_j^{N}(\sigma_i^{j, x}): x
  \leq M, \, j < N,\, \theta_j^{k_n}(\sigma_i^{j, x}-) \leq T
  \}$. Define $\lambda_n$ such that for every $s \in S_n$:
  \begin{align*}
    \lambda_n(\theta_{N+1}^{k_n} (s)) &= \theta_{N+1}^\infty(s).
  \end{align*}
  Complete $\lambda_n$ linearly between these points and let it evolve
  linearly with angular coefficient $1$ after the last point.  From the
  facts that $S_n$ is finite and $\lambda_n$ is linear by parts, it readily
  follows that $\lambda_n$ is Lipschitz continuous. Theorem
  \ref{thm:nao-trivial} guarantees that $\lambda_n$ is strictly
  increasing, so it qualifies as a candidate for temporal distortion.

  Using Remark \ref{obs:coordinate-k-process}, we know that the $j$-th
  coordinate, $j < N$, of $X_{k_n}(t)$ is equal to an $x \leq M$ if
  and only if:
  \begin{align*}
    t \in \bigcup_{i=1} \left[
      \theta_j^{k_n} (\sigma_i^{j, x} -),
      \theta_j^{k_n} (\sigma_i^{j, x})
    \right) =
    \bigcup_{i=1} \left[
      \theta_{N+1}^{k_n} ( \theta_j^N (\sigma_i^{j, x} -)),
      \theta_{N+1}^{k_n} ( \theta_j^N (\sigma_i^{j, x}))
    \right) \Leftrightarrow\\
    \lambda_n(t) \in
    \bigcup_{i=1} \left[
      \theta_{N+1}^{\infty} ( \theta_j^N (\sigma_i^{j, x} -)),
      \theta_{N+1}^{\infty} ( \theta_j^N (\sigma_i^{j, x}))
    \right) =
    \bigcup_{i=1} \left[
      \theta_{j}^{\infty} (\sigma_i^{j, x} -),
      \theta_{N+1}^{\infty} (\sigma_i^{j, x})
    \right).
  \end{align*}

  Note that we have used Corollary \ref{cor:clock-composition} in the last
  passage. With this we conclude that, for any coordinate $j < N$:
  \begin{align*}
    \sup_{t \in [0, T]} \rho_0(X_{k_n, j}(t), Y_j(\lambda_n(t))) &\leq
    \frac{1}{M+1} < \frac{\epsilon}{2},\\
    \sup_{t \in [0, T]} \rho(X_{k_n}(t), \sY(\lambda_n(t))) &<
    \frac{\epsilon}{2} + \sum_{j \geq N} \frac{1}{2^{j}} < \epsilon.
  \end{align*}

  This concludes \eqref{eq:conv-infinite-process-1}. To show
  \eqref{eq:conv-infinite-process-2} take $T^\prime > 0$ such that
  $\theta_j^\infty(T^\prime) > T+1$ for every $j < N$, and note that,
  for large enough $n$:
  \begin{displaymath}
    \sup_{t \in [0, T]} |\lambda_n(t) - t| =
    \max_{s \in S_n} \left|
      \theta_{N+1}^{k_n}(s) - \theta_{N+1}^{\infty}(s)
    \right|
    \leq
    \sup_{s \in [0, T^\prime]} \left|
      \theta_{N+1}^{k_n}(s) - \theta_{N+1}^{\infty}(s)
    \right|
    \xrightarrow{n \to \infty} 0.
    \qedhere
  \end{displaymath}
\end{proof}

\begin{observacao}
 \label{rmk:conds}
 We come back to the conditions~(\ref{eq:constant-choice}) 
  and~(\ref{cond2}) to discuss what if any of them is missing. Without~(\ref{eq:constant-choice}),
  we cannot insure the existence of $W(\cdot)$ and thus of the $\theta^\infty_k$'s, without 
  which we cannot define the limiting K process. One might try to obtain convergence in distribution 
  of the clocks, but we do not see an approach to, say, take the limit of the Laplace transform of 
  $\theta^n_k$ as $n\to\infty$. If we keep~(\ref{eq:constant-choice}) but try to relax~(\ref{cond2}),
  then it is the nontriviality of the $\theta^\infty_k$'s that is at stake: we know or presume that
  it is identically zero in this case, and this disables the definition of a meaningful limiting 
  K process.
  We would need in this case to rescale $\theta^n_k$ before taking the limit, but it is not clear 
  to us even 
  which would the right scale be, or if this would lead to the definition of the correct limit
  for the K process. If this could be done, the limiting process could be quite
  different from the one we obtained above.
\end{observacao}



%
\section{Empirical Measure}
\label{sec:invar-meas-some}

In this section we will assume that $\underline\gamma$ is fixed. All
results from this section are valid for almost all choices for this
random environment (under nontriviality conditions).

The main result of this section is the computation of the asymptotic 
empirical measure of
the K process on a tree with infinite depth, that is, the proportion
of the time that this process spends on cylinders $[x|_k] = \{
y|_\infty \in \widebar{\sN}_*^{\sN_*}: y|_k = x|_k \}$. We will show
that it is almost surely given by:
\begin{equation}
  \label{eq:medida-invariante}
  \lim_{t \to \infty} \frac{1}{t} \int_0^\infty \ind\left\{
    \sY(t) \in [x|_k]
  \right\} d t
  =
  \frac{
    \widebar{\gamma}_k(x|_k) \sE_\gamma \left[
      \theta_{x|_k}^\infty(1) \right]
  }{\sE_\gamma\left[\theta_1^\infty(1) \right]}.
\end{equation}

We showed that the right hand side of this last expression is well-defined
\emph{a.s.}~in Theorem \ref{thm:limit-clock} (since in particular we are
under nontriviality conditions). To compute the empirical
measure, we will rely on the following assumption.
\begin{suposicao}
  \label{conj:markov}
  The K process on a tree with finite depth is strongly Markovian.
\end{suposicao}

A proof of the Markov property of the finite level K process can be found
in~\cite{kn:g}. The Feller property (as well as the Markov property itself) 
should follow from arguments similar to those for the 1 level case used 
in~\cite{kn:fm}, establishing the strong Markov property. We choose to 
not go to detail, and leave the issue as an assumption. 




Before computing the empirical measure, let us prove some auxiliary
results.

\begin{proposicao}
  \label{prop:esperanca-linear}
  If $T$ is a positive random variable independent from
  $\theta_{x|_k}^\infty$, then
  \begin{displaymath}
    \sE_\gamma \left[
      \theta_{x|_k}^\infty(T)
    \right]
    = \sE_\gamma\left[\theta_{x|_k}^\infty(1) \right] \sE_\gamma(T).
  \end{displaymath}
\end{proposicao}

\begin{proof}
  We will prove this result only for $\theta_1^\infty$, Remark
  \ref{obs:pedacos-clock} extends the result to
  $\theta_{x|_k}^\infty$.

  Fix an arbitrary $n, m \in \sN_*$, since $\theta_1^\infty$ is a
  subordinator (Remark \ref{obs:theta-inf-subordinador}) then
  $\theta_1^\infty(n/m)$ has the same law then the sum of $n$
  independent copies of $\theta_1^\infty(1/m)$. By the same argument
  $\theta_1^\infty(1)$ has the same law as the sum of $m$ independent
  copies of $\theta_1^\infty(1/m)$. Therefore:
  \begin{displaymath}
    \sE_\gamma \left[\theta_1^\infty\left(\frac{n}{m}\right)\right] =
    n \sE_\gamma \left[\theta_1^\infty\left(\frac{1}{m}\right)\right] =
    \frac{n}{m} \sE_\gamma \left[\theta_1^\infty\left( 1 \right)\right]
  \end{displaymath}

  For an arbitrary real $t > 0$, take $q, r$ rational numbers such
  that $0 < q < t < r$. Since a subordinator is monotonic then:
  \begin{gather*}
    \sE_\gamma \left[\theta_1^\infty\left( q \right)\right] \leq
    \sE_\gamma \left[\theta_1^\infty\left( t \right)\right] \leq
    \sE_\gamma \left[\theta_1^\infty\left( r \right)\right]\\
    \Leftrightarrow
    q \sE_\gamma \left[\theta_1^\infty\left( 1 \right)\right] \leq
    \sE_\gamma \left[\theta_1^\infty\left( t \right)\right] \leq
    r \sE_\gamma \left[\theta_1^\infty\left( 1 \right)\right].
  \end{gather*}

  Since $q, r$ were taken arbitrarily, we have that $ \sE_\gamma
  \left[\theta_1^\infty\left( t \right)\right] = t
  \left[\theta_1^\infty\left( 1 \right)\right]$. To complete the
  proof, let $\nu$ be the probability measure associated with $T$.
  Since $T$ is independent from $\theta_1^\infty$, we conclude that
  \begin{equation}\nonumber
    \sE_\gamma \left[ \theta_1^\infty(T) \right]
    =
    \int \sE_\gamma \left[ \theta_1^\infty(t) \right] \nu(d t)
    =
    \sE_\gamma \left[ \theta_1^\infty(1) \right] \int t \nu(d t)
    =
    \sE_\gamma \left[ \theta_1^\infty(1) \right] \sE_\gamma(T). 
  \end{equation}
\end{proof}

The next result states that, for a fixed $t \geq 0$, the family
$\{\sE_\gamma[\theta_{x|_k}^\infty(t)]: k \in \sN_* , \, x|_k \in \sN_*^k\}$
obeys a composition law analogous as the one stated in Proposition
\ref{prop:composition-W} for the family $\{W(x|_k): k \in \sN_* , \,
x|_k \in \sN_*^k\}$.

\begin{proposicao}
  \label{prop:comp-exp-theta}
  For any fixed $t > 0$ and $x|_k \in \sN_*^k$, if $\sum_i(1-\alpha_i)
  < \infty$:
  \begin{equation}
    \label{eq:comp-exp-theta-1}
    \sE_\gamma \left[ \theta_{x|_k}^\infty (t)  \right] =
    \sum_{x_{k+1}} \gamma_{k+1} (x|_{k+1})
    \sE_\gamma \left[
      \theta_{x|_{k+1}}^\infty (t)
    \right].
  \end{equation}
  Furthermore:
  \begin{equation}
    \label{eq:comp-exp-theta-2}
    \sE_\gamma \left[ \theta_{x|_k}^\infty (\sigma_1^{k+1, x_{k+1}} - )
    \right] =
    \sE_\gamma \left[
      \theta_{x|_{k}}^\infty (1)
    \right]
    - \gamma_{k+1} (x|_{k+1})
    \sE_\gamma \left[
      \theta_{x|_{k+1}}^\infty (1)
    \right].
  \end{equation}
\end{proposicao}

\begin{proof}[Sketch of proof]
  To prove \eqref{eq:comp-exp-theta-1}, apply Corollary
  \ref{cor:clock-composition} to break the contribution from each
  $x_{k+1}$ to $\theta_{x|_k}^\infty$, and then use Proposition
  \ref{prop:esperanca-linear}.

  To prove \eqref{eq:comp-exp-theta-2}, note that
  $\theta_{x|_k}^\infty$, up to time $\sigma_1^{k+1, x_{k+1}}$ is
  independent from this time, with the exception that it does not see
  any point of the Poisson Process $N^{k+1, x_{k+1}}$, which would be
  equivalent of setting $\gamma_{k+1}(x|_{k+1}) = 0$. Denoting by
  $\widehat{\theta}_{x|_k}^\infty$ a version of $\theta_{x|_k}^\infty$
  in which this modification was made, we can use the previous result
  and Proposition \ref{prop:esperanca-linear} to obtain:
  \begin{align*}
    \sE_\gamma \left[
      \theta_{x|_k}^\infty (\sigma_1^{k+1, x_{k+1}} - )
    \right]
    &=
    \sE_\gamma \left[
      \widehat{\theta}_{x|_k}^\infty (\sigma_1^{k+1, x_{k+1}} )
    \right]\\
    &=
    \sE_\gamma(\sigma_1^{k+1, x_{k+1}})
    \sE_\gamma \left[
      \widehat{\theta}_{x|_k}^\infty (1)
    \right]\\
    &=
    \sum_{y \neq x_{k+1}} \gamma_{k+1}(x|_k y) \sE_\gamma \left[
      \theta_{x|_k y}^\infty (1)
    \right]\\
    &=
    \sE_\gamma \left[
      \theta_{x|_k}^\infty (1)
    \right] -
     \gamma_{k+1}(x|_{k+1}) \sE_\gamma \left[
      \theta_{x|_{k+1}}^\infty (1)
    \right]
    \qedhere
  \end{align*}
\end{proof}

\begin{definicao}
  Let us denote the first $k$ coordinates of the process $\sY$ by
  $Y|_k$, that is:
  \begin{displaymath}
    Y|_k := (Y_1, \ldots, Y_k)
  \end{displaymath}
\end{definicao}

\begin{definicao}
  For a fixed $y|_k \in \sN_*^k$, let us denote by $U_i$ and $V_i$ the
  $i$-th entrance and exit times respectively of $\sY$ in
  $[y|_k]$. That is, we define $V_0 := 0$ and for $i = 1, 2, \ldots$:
  \begin{subequations}
    \label{eq:entrance-exit-times}
    \begin{align}
      \label{eq:entrance-times}
      U_i &:= \inf \left\{
        t > V_{i-1}: Y|_k(t) = y|_k
      \right\}\\
      \label{eq:exit-times}
      V_i &:= \inf \left\{
        t > U_i: Y|_k(t) \neq y|_k
      \right\}
    \end{align}
  \end{subequations}
\end{definicao}

\begin{observacao}
  Since $\sY$ is right continuous, then $Y|_k(U_i) = y|_k$ and
  $Y|_k(V_i) \neq y|_k$. Furthermore it is true that $Y_j(U_i) =
  \infty$ for every $j > k$ and that:
  \begin{displaymath}
    Y_j(V_i) = \begin{cases}
      y_j & \text{ if } j < k,\\
      \infty & \text{ otherwise.}
    \end{cases}
  \end{displaymath}
\end{observacao}

The increment $V_i - U_i$ is the time spent by $\sY$ on $y|_k$ on its
$i$-th visit. It is equal to $\theta_k^\infty(\sigma_j^{k, y_k}) -
\theta_k^\infty(\sigma_j^{k, y_k}-)$ for some $j$. Therefore $\sE_\gamma(V_i
- U_i) = \sE_\gamma\left[ \theta_{y|_k}^\infty\left(\gamma_k(x|_k) T_1^{k,
      y_k}\right) \right] = \gamma_k(y|_k)
\sE_\gamma(\theta_{y|_k}^\infty(1))$.

\begin{proposicao}
  \label{prop:time-outside}
  The increment $U_{i+1} - V_i$ is the time spent outside of $[y|_k]$
  between successive visits to this cylinder. Its expected value can
  be computed as:
  \begin{equation}
    \label{eq:time-outside}
    \sE_\gamma (U_{i+1} - V_i) = \frac{\sE_\gamma\left( \theta_1^\infty (1)
      \right)}{\widebar{\gamma}_{k-1}(y|_{k-1})} -
    \gamma_k(y|_k) \sE_\gamma \left( \theta_{y|_k}^\infty (1) \right),
  \end{equation}
  under the convention that $\widebar{\gamma}_0(y|_0) := 1$.
\end{proposicao}

\begin{proof}
  Let us denote by $S_j^1, S_j^2, \ldots$ the times between visits to
  $y|_j$. We claim that it follows from Assumption~\ref{conj:markov}
  that this variables form an \emph{i.i.d.}~sequence. Indeed we can
  write the cycles determined by the successive visits of $Y|_j$ to
  $y|_j$ in terms of the cycles determined by the successive visits of
  $X_j$ to $y|_j$, where $X_j$ is the j level K process constructed in
  Section~\ref{sec:constr-main-result}. Each former cycle is obtained
  as the sum over constancy intervals of $X_j$ within the
  corresponding latter cycle of
  $\theta^\infty_{x|_j}(L_j(x|_j, b_j)) -
  \theta^\infty_{x|_j}(L_j(x|_j, a_j))$,
  where $x|_j$ is the constant value of $X_j$ within the respective
  constancy interval $[a_j, b_j)$. The strong Markov property of $X_j$
  implies that the distribution of constancy intervals of $X_j$ within
  the latter cycles are \emph{i.i.d.}~when we vary those cycles.
  Since moreover
  $\theta^\infty_{x|_j}(L_j(x|_j, b_j)) -
  \theta^\infty_{x|_j}(L_j(x|_j, a_j))$
  are independent when we vary the constancy intervals (see Remark
  \ref{obs:pedacos-clock} and Lemma \ref{lema:theta-subordinator}
  above),
  the claim follows. (One ought also to be able to make an argument
  for the claim, dispensing with Assumption~\ref{conj:markov}, by
  using directly the structure of the model together with the lack of
  memory of the exponential distribution.)

  Let us now define $a_j := \sE_\gamma(S^1_j)$.  We want to compute $a_k$.

  At time $V_i$, the process just exited $y|_k$, that is, $V_i =
  \theta_k^\infty(\sigma_{i^\prime}^{k, y_k})$ for some
  $i^\prime$. There exists an $\sigma_{i^{\prime\prime}}^{k-1,
    y_{k-1}}$ such that $\sigma_{i^\prime}^{k, y_k} \in
  [\clock_{k-1}(\sigma_{i^{\prime\prime}}^{k-1, y_{k-1}} -),
  \clock_{k-1}(\sigma_{i^{\prime\prime}}^{k-1, y_{k-1}}))$. This
  interval has length $\gamma_{k-1}(y|_{k-1})
  T_{i^{\prime\prime}}^{k-1, y_{k-1}}$. Because of the loss of memory
  of the exponential distribution, the distribution of
  $\clock_{k-1}(\sigma_{i^{\prime\prime}}^{k-1, y_{k-1}}) -
  \sigma_{i^\prime}^{k, y_k}$ is exactly the same as the distribution
  of the whole interval.

  With probability $p = \gamma_{k-1}(y|_{k-1})/(1+
  \gamma_{k-1}(y|_{k-1}))$ it will happen that $\sigma_{i^\prime +
    1}^{k, y_k} < \clock_{k-1}(\sigma_{i^{\prime\prime}}^{k-1,
    y_{k-1}}))$. In this case the K process will visit $y|_k$ again
  before exiting $y|_{k-1}$.

  If this does not happen it will take a time $S_{k-1}^1$ for the
  process to visit $y|_{k-1}$ again, after that it will have a
  probability $p$ of visiting $y|_k$ during this visit, if this does
  not happen then it will take a time $S_{k-1}^2$ for a third try, and
  so on.

  Therefore the time spent outside of $y|_{k-1}$ has the same law as
  $\sum_{i=1}^M S_{k-1}^i$, where $M$ is a geometric random variable,
  with success probability $p$, independent of $S_{k_1}^i, i=1, 2,
  \ldots$.

  The total time spent on $y|_{k-1}$ but outside of $y|_k$ has the
  same law as $\theta_{y|_{k-1}}(\sigma_1^{k, y_k}-)$. Therefore:
  \begin{align*}
    a_k &= a_{k-1} \frac{1-p}{p} + \sE_\gamma \left[
      \theta_{y|_{k-1}}(\sigma_1^{k, y_k}-)
    \right]\\
    &= \frac{a_{k-1}}{\gamma_{k-1}(y|_{k-1})} +
    \sE_\gamma \left[
      \theta_{y|_{k-1}}^\infty (1)
    \right]
    - \gamma_{k} (y|_{k})
    \sE_\gamma \left[
      \theta_{y|_{k}}^\infty (1)
    \right].
  \end{align*}
  Applying an induction in $k$, we obtain \eqref{eq:time-outside}.
\end{proof}

It follows from Assumption~\ref{conj:markov}, as argued above, 
that increments between
$(U_i, V_i)$ are independent from each other as $i$ varies. Therefore we can 
use the strong law of large numbers to compute the proportion of the time that
the process stays on the state $y|_k$ as:
\begin{align*}
  \pi(y|_k) &:=
  \lim_{n \to \infty} \frac{\sum_{i = 1}^\infty V_i - U_i}{V_n}
  \substack{\emph{a.s.}\\=}
 \frac{\sE_\gamma \left[ V_i - U_i \right]}{\sE_\gamma \left[ U_{i+1} - U_i
    \right]}\\
  &= \frac{\sE_\gamma \left[ V_i - U_i \right]}{\sE_\gamma \left[ U_{i+1} - V_i
    \right] + \sE_\gamma \left[V_i - U_i \right]}
  \notag
  = \frac{\widebar{\gamma}_k(y|_k) \sE_\gamma\left[ \theta_{y|_k}^\infty (1)
    \right]}{\sE_\gamma \left[ \theta_1^\infty(1) \right]},
\end{align*}
and formula~\eqref{eq:medida-invariante} is established.
Proposition \ref{prop:comp-exp-theta} can be then used to verify the
conditions of Kolmogorov's Consistency Theorem, concluding that this
function $\pi$ defines a probability measure on the product
$\sigma$-algebra. It should be the equilibrium measure for the 
infinite level GREM-like K process.

We finish with the remark that in case the upper bound for 
$\sE_\gamma\left[\theta_{x|_k}^\infty(1)
\right]$ in Theorem \ref{thm:limit-clock} saturated (that
would be the case if the convergence $\theta_\cdot^n(1)\to\theta_\cdot^\infty(1)$
took place in $L_1$), then $\pi$ would
have a more explicit, nicer looking form, namely 
$$\pi(y|_k)=\frac{\widebar{\gamma}_k(y|_k) W(y|_k)}{W(\emptyset)}.$$
At the moment, we do not know whether this is the case or not.



\section*{Acknowledgements}

The authors would like to thank NUMEC-USP for the hospitality shown to the second author
during the development of this work.

%



\begin{thebibliography}{xxxxxx 89}

\bibitem{kn:fgg} Fontes, L.R.G.; Gava, R.J.; Gayrard, V.~(2014)\\
The K-process on a tree as a scaling limit of the GREM-like trap model\\
{\em Ann. Appl. Probab.}~{\bf 24}, 857-897

\bibitem{kn:d2} Derrida, B.~(1985) \\
A generalization of the random energy model which includes correlations between energies\\
{\em J.~Physique Lett.}~{\bf 46}, L401-L407


\bibitem{kn:d1} Derrida, B.~(1980)\\ 
Random-energy model: Limit of a Family of Disordered Models\\
{\em Phys.~Rev.~Lett.}~{\bf 45}, 79-82

\bibitem{kn:sn} Sasaki, M.; Nemoto, K.~(2000)\\
Analysis on Aging in the Generalized Random Energy Model\\
{\em J.~Phys.~Soc.~Jpn.}~{\bf 69}, 3045-3050  


\bibitem{kn:b} 
Bouchaud, J.-P. (1992)\\
Weak ergodicity breaking and aging in disordered systems\\
\emph{J.~Phys.~I France}~{\bf 2}, 1705–1713



\bibitem{kn:abg} Ben Arous, G.; Bovier, A.; Gayrard, V.~(2003)\\
Glauber dynamics of the random energy model. II. Aging below the critical temperature\\
{\em Comm.~Math.~Phys.}~{\bf 236}, no.~1, 1--54. 



\bibitem{kn:fl} Fontes, L.R.G.; Lima, P.H.S.~(2009)\\
Convergence of symmetric trap models in the hypercube\\
In: XVth International Congress on Mathematical Physics, 2006, Rio de Janeiro\\
{\em New Trends in Mathematical Physics}, Springer, 285-297
   
\bibitem{kn:fm} Fontes, L.R.G.; Mathieu, P.~(2008)\\
K-processes, scaling limit and aging for the trap model on the complete graph\\
{\em Ann.~Probab.}~{\bf 36}, 1322-1358.





              


\bibitem{kn:abc} Ben Arous, G.; Bovier, A.; \v Cern\'y, J.~(2008)\\		
Universality of the REM for dynamics of mean-field spin glasses\\
{\em Comm.~Math.~Phys.}~{\bf 282}, 663-695

\bibitem{kn:bg} Bovier, A.; Gayrard, V.~(2013)\\
Convergence of clock processes in random environments and ageing in the $p$-spin SK model\\
{\em Ann.~Probab.}~{\bf  41}, 817-847.



\bibitem{kn:fg} Fontes, L.R.G.; Gayrard, V.\\
Asymptotic behavior of a hierarchical hopping dynamics for the two level GREM at low temperature and ergodic time scale\\
in preparation







\bibitem{kn:d} Durrett, R.~(2010)\\
Probability: theory and examples\\
{\em  Cambridge}, 4th ed.

\bibitem{kn:st} Samorodnitsky, G.; Taqqu, M.~(1994)\\
Stable Non-Gaussian Random Processes: Stochastic Models with Infinite Variance\\
{\em  Chapman and Hall}

\bibitem{kn:l} Laue, G.~(1980)\\
Remarks on the relation between fractional moments and fractional derivatives of characteristic functions\\
{\em J.~Appl.~Probab.}~{\bf 17}, 456-466




\bibitem{kn:ek} Ethier, S.N.; Kurtz, T.G.~(1986)\\
Markov processes. Characterization and convergence\\ 
{\em  Wiley}

\bibitem{kn:k} Kingman, J.F.C.~(1993)\\
Poisson processes\\
{\em  Oxford}

    

\bibitem{kn:g} Gava, R.J.~(2011)\\
Scaling limits of trap models on a tree\\
Ph.D.~thesis; University of São Paulo (in portuguese)

          


\end{thebibliography}

\end{document}